\documentclass[a4paper,10pt]{amsart}
\usepackage{amssymb}
\usepackage{appendix}
\usepackage[symbol]{footmisc}
\usepackage{a4wide}
\usepackage{graphicx}
 \usepackage[all]{xy}
 
\usepackage{xcolor}
\usepackage{boxhandler}
\usepackage{caption}
\usepackage{lipsum}
\newtheorem*{theorem*}{Main Theorem}
\newtheorem{theorem}{Theorem}

\newtheorem{lem}{Lemma}
\newtheorem{prop}[theorem]{Proposition}
\newtheorem{rem}[theorem]{Remark}

\theoremstyle{definition}

\newtheorem*{ques}{Question}
\newtheorem*{fact}{Fact}

\newcommand{\Leb}{\mathop{\mathrm{Leb}}}
\newcommand{\vertiii}[1]{{\left\vert\kern-0.25ex\left\vert\kern-0.25ex\left\vert #1 
    \right\vert\kern-0.25ex\right\vert\kern-0.25ex\right\vert}}
\theoremstyle{remark}

\numberwithin{equation}{section}

\begin{document}

\title[ SRB measures for partially hyperbolic systems]{  SRB measures for partially hyperbolic systems with one-dimensional center subbundles.}

\author{David Burguet}

\date{October 2023}


 \maketitle
\begin{Large}

\begin{abstract}
For a partially hyperbolic attractor with a center bundle splitting in a dominated way into one-dimensional subbundles  we show that for Lebesgue almost every point there is an  empirical measure from $x$ with a  SRB component.  Moreover if the center exponents are non zero,  then  $x$ lies in the basin of an ergodic hyperbolic SRB measure and there are only finitely many such measures. This gives another proof of the existence of SRB measures in this context,  which was   established firstly  in \cite{dawei} by using random perturbations. Moreover this generalizes  results of \cite{SDJ,yang} which deal with a single one-dimensional center subbundle.   \end{abstract}

\vspace{1cm}

  For a $\mathcal C^{1+}$ diffeomorphism $f$  on a compact smooth manifold $\mathbf M$,  an invariant measure is called an SRB (Sina\"i-Ruelle-Bowen) measure when it has a positive Lyapunov exponent almost everywhere and satisfies Pesin entropy formula,  i.e.  its entropy is equal to the sum of the positive Lyapunov exponents.  Equivalently the conditional measures along local unstable Pesin manifolds are absolutely continuous with respect to the Lebesgue measure \cite{Led}.  A fundamental problem in dynamics consists in understanding the statistical behaviour of $(\mathbf M,f)$ :  what are the limits for the weak-* topology  of the empirical measures $\mu_x^n:=\frac{1}{n}\sum_{0\leq k<n}\delta_{f^kx}$, $x\in \mathbf M$, when $n$ goes to infinity?  Ergodic SRB  measures $\mu$,   which are hyperbolic (i.e.  with non zero Lyapunov exponents), are specially important in this respect,  because their basins $\mathcal B(\mu)=\{x\in \mathbf M, \ \mu_x^n\xrightarrow{n\rightarrow +\infty}\mu\}$ have positive Lebesgue measure \cite{Led84}.  For hyperbolic attractors,    Sina\"i,  Ruelle and Bowen have studied these measures and their basin in the seventies.   In this setting,   there are finitely may ergodic (hyperbolic) SRB  measures whose basins cover a set of full Lebesgue measure \cite{Sin72,BoR75}.   
  
  Beyond unifrom hyperbolicity,  existence of SRB measures has been established for partially hyperbolic attractors whose center bundle splits in a dominated way into one-dimensional subbundles by using random perturbations or unstable entropies \cite{Young, dawei,SDJ}.  In this note we give another proof of this result by using an entropic variant  of the geometrical method developped in  \cite{Bur,bd}.  Moreover we show that  at Lebesgue almost every point $x$ there is a limit $\mu$ of $\left(\mu_x^n\right)_n$ such that some ergodic component of $\mu$ is an SRB measure. 

In the following  we consider an attracting  set $\Lambda$ of a $\mathcal C^{1+}$ diffeomorphism $f$, i.e. $\Lambda$ is a compact invariant set with an open neighborhood $U\subset \mathbf  M$ satisfying $f(\overline{U})\subset U $ and $\Lambda=\bigcap_{n \in \mathbb N}f^n U$, with  a partially hyperbolic splitting $T\mathbf M|_{\Lambda} = E^u \oplus_{\succ}  E_{1} \oplus_{\succ}\cdots \oplus_{\succ} E_{k} \oplus_{\succ} E^s$ with $\mathrm{dim}(E_{i})=1$ for $i=1,\cdots, k$.   The invariant bundles $E^u$ and $E^s$ are expanded and contracted respectively :   there are $C>0$ and $\lambda\in ]0,1[$ such that for any $x\in\Lambda$ and any $n\in\mathbb N$, it holds that  $\|Df^n|_{E^s(x)}\|\le C\lambda^n$ and $\|Df^{-n}|_{E^u(x)}\|\le C\lambda^n$. Moreover for  two $Df$-invariant subbundles $E,F\subset T_\Lambda \mathbf  M$, the bundle $E$ is \emph{dominated} by $F$, denoted as $F\oplus_\succ E$,   when there are $C>0$ and $\lambda\in ]0,1[$ such that $\|Df^n|_{E(x)}\|\|Df^{-n}|_{F(f^nx)}\|\le C\lambda^n$ for any $x\in\Lambda$ and any $n\in\mathbb N$. Any diffeomorphim $\mathcal C^1$ away from the set of diffeomorphisms exhibiting a homoclinic tangency may be approximated by partially hyperbolic diffeomorphisms of this form (see \cite{crov,dawei}). 

 An empirical measure from $x\in \mathbf M$ is a limit for the weak-$*$ topology of the sequence of atomic measures $(\mu_x^n)_{n\in \mathbb N}$.   We let $pw(x)$ be the set of  empirical measures from $x$.
When $\nu$ and $\mu$ are two non-zero Borel  measures on $\mathbf M$,  we say that $\nu$ is a component of $\mu$  when $\nu(A)\leq \mu(A)$ for any Borel set $A$.  An SRB component of an invariant measure $\mu$  is a component of $\mu$,  such that the associated probability $\frac{\nu(\cdot)}{\nu(\mathbf M)}$ is an SRB probability measure. \\

The main results of this paper read as follows. 

\begin{theorem}\label{mmain}
With the above notations, for Lebesgue almost every $x\in U$,  we have the following dichotomy :
\begin{itemize}
\item either $x$ lies in the basin of an ergodic hyperbolic SRB measure, 
\item or there is $\mu\in pw(x)$ with non-hyperbolic SRB components. 
\end{itemize}
\end{theorem}
Theorem \ref{mmain} with $k=1$ has been proved in \cite{SDJ,yang} with another method.  
 In the second case of the alternative,  the empirical measures from $x$ may not converge - the point $x$ is said to have historical behaviour (see Theorem B in \cite{SDJ}). 
 
\begin{theorem}\label{ttwo}Assume moreover  that any ergodic SRB measure  is  hyperbolic. Then there are finitely many ergodic  hyperbolic SRB measures, whose basins cover a set of full Lebesgue measure in $U$. 
\end{theorem}

Theorem \ref{ttwo} is proved in \cite{caomi} under the stronger assumption that all  ergodic Gibbs $u$-states are hyperbolic.  If all SRB measures are hyperbolic, the second case in Theorem \ref{mmain}  never occurs. Thus the topological basin $U$ is Lebesgue almost covered by the basins of ergodic SRB measures, so that the main new point in Theorem \ref{ttwo} is the finiteness property of SRB measures. \\

To build SRB measures we estimate the entropy of limit empirical measures from below by using a Gibbs property at hyperbolic times  as in \cite{bd}.  By considering empirical measures on a  specific subset of times we may  ensure these measures are in fact SRB.  After recalling some standard properties of empirical measures,  we introduce in the second section different notions of hyperbolic times and we relate their density with hyperbolic properties of the limit empirical measures.  In the third section we deduce  Theorem \ref{mmain} and Theorem \ref{ttwo} respectively  from Proposition \ref{pro:mar} and Proposition \ref{prop:second},   which describe more precisely the statistical behaviour in terms of the densities of hyperbolic times.  The proofs of these two propositions,  which 
share some similarities, are given in the last two sections. 

\newpage

\tableofcontents

\section{Empirical measures}

\subsection{General setting}
  We consider now  a general  invertible topological system $(X,T)$ (i.e. $T$ is a homeomorphism of a compact metric space $X$) together with a continuous  real observable $\phi:X\rightarrow \mathbb R$.  We let $\mathcal M(X)$ be the compact set of Borel probability measures endowed  with the weak-$*$ topology and we consider the  compatible convex  distance $\mathfrak d
$ on $\mathcal M(X)$ associated to a  countable family, $\{f_n, \, n\in \mathbb N\}$,  dense in the set $\mathcal C(X)$ of continuous real functions on $X$ endowed with the uniform topology : 
$$\forall \mu,\nu\in \mathcal M(X), \ \mathfrak d(\mu, \nu):=\sum_{n\in \mathbb N}\frac{|\int f_n\,d\mu-\int f_n\,d\nu|}{2^n(1+\|f_n\|_\infty)}.$$
A measure $\mu \in \mathcal M(X)$ is said $M$-almost invariant,  $M>0$,  when $\mathfrak d \left(\mu, T_*\mu\right)\leq \frac{1}{M}$.  Let $\mathcal M(X,T)$ be the compact subset of $\mathcal M(X)$ given by $T$-invariant measures. 

For $x\in X$ and $n\in \mathbb N$ we denote by $\mu_x^n$ the usual empirical measure 
$$\mu_x^n:=\frac{1}{n}\sum_{0\leq k<n}\delta_{T^kx}.$$ We also recall that  $pw(x)$ denotes the compact subset of $\mathcal M(X,T)$ given by the limits of $(\mu_x^n)_n$.   For a subset $E$ of $\mathbb N$ we let  $\overline{E}$ be the complement set of $E$, i.e. $\overline{E}:=\mathbb N\setminus E$.     We consider the empirical measure $\mu_x^{n}[E]$  associated to $E$: 
$$\mu_x^{n}[E]=\frac{1}{n}\sum_{0\leq k<n, \ k\in E}\delta_{T^kx}$$
For a subset $E$ of $\mathbb N$ and $M\in \mathbb N$ we let $E(M)=\{k\in \mathbb N,\ k+m\in E \text{ for some } 1\leq m\leq M \}$. 
The empirical measures $\left(\mu_x^{n}[E(M)]\right)_n$,  therefore their limits in $n$ or linear combinations,  are $M$-almost invariant.
\subsection{Bounding from below the entropy of empirical measures}

Following Misiurewicz's proof of the variational principle, we estimate the entropy of empirical measures from below. For a finite partition $P$ of $X$ and  a finite subset $F$ of $\mathbb N$, we let $P^F$ be the iterated partition $P^F=\bigvee_{k\in F}T^{-k}P$. When $F=\left[ 0,n\right[$, $n\in \mathbb N$, we just let $P^{F}=P^n$. We denote by $P(x)$ the element of $P$ containing $x\in X$.  Given a  measure $\mu$ on $X$ we let $\mu^{F}:=\frac{1}{\sharp F}\sum_{k\in F} T_*^k\mu$. 

For a Borel probability measure $\mu$ on $X$, the static entropy $H_\mu(P)$ of $\mu$ with respect to a (finite measurable) partition $P$  is defined as follows:
\begin{align*}
H_\mu(P)&=-\sum_{A\in P}\mu(A)\log \mu(A),\\
&=-\int \log \mu\left(P(x)\right)\, d\mu(x).
\end{align*}
When $\mu$ is $T$-invariant, we recall that  the measure theoretical entropy of $\mu$ with respect to $P$ is then   $$h_\mu(P)=\lim_n\frac{1}{n}H_{\mu}(P^n)$$ 
and the entropy $h(\mu)$ of $\mu$ is 
$$h(\mu)=\sup_Ph_\mu(P).$$

\begin{lem}\label{lem:Miss}\cite{Bur}
With the above notations we have 
\begin{align*}\forall m\in \mathbb N^*,\ \frac{1}{m}H_{\mu^F}(P^m)&\geq  \frac{1}{\sharp F}H_{\mu}(P^{F})-3m\log \left(\sharp P\right)\frac{\sharp \partial F}{\sharp F}.
\end{align*}
\end{lem}

In the above statement the set of times $F$ is a fixed finite subset.  Here we need to work with measurable finite set-valued maps $F$.  In this context we define $\mu^F$ as follows 
$$\mu^{F}:=\frac{\int\sum_{k\in F(x)} \delta_{T^kx} \, d\mu(x)}{\int \sharp F(x)\, d\mu(x)}$$ 
We may generalize Lemma \ref{lem:Mis} as follows:

\begin{lem}\label{lem:Mis}
With the above notations we have  for all $ m\in \mathbb N^*$:
\begin{align*} \frac{\int \sharp F(x)\, d\mu(x)}{m}H_{\mu^F}(P^m)&\geq \int-\log \mu\left(P^{F(x)}(x)\right) \, d\mu(x)\\
& -H_\mu(F)-\int 3m \sharp \partial  F(x)\log\left( \sharp  P\right) \, d\mu(x).
\end{align*}
\end{lem}
\begin{proof}We recall that for a given $x\in X$,  the set $P^{F(x)}(x)$ denotes the atom of the iterated partition $P^{F(x)}$ which contains $x$.  The collection of sets $P^{F(x)}(x)$, $x\in X$,  does not a priori form a partition of $X$.    We denote also by $F$ the partition  associated to the distribution of $F$ and by $P^F$ the partition finer than $F$ whose atoms are the sets of the form $\{x\in X, \   F(x)=E \text{ and } x\in A\}$ for a subset of integers $E$ and an atom $A$ of $P^E$.  Then the atom $P^F(x)$ of $P^F$ containing $x$  is a subset of  $P^{F(x)}(x)$,  therefore 
\begin{align*}
\int-\log \mu\left(P^{F(x)}(x)\right) \, d\mu(x)& \leq 
\int-\log \mu\left(P^{F}(x)\right) \, d\mu(x),\\
&\leq H_\mu(P^F). 
\end{align*}
  We let $F_E$ be the atom of $F$ given by $F_E:=\{x\in X, \ F(x)=E\}$.   By conditioning with respect to $F$ we get with $\mu_{F_E}:=\frac{\mu(F_E\cap \cdot)}{\mu(F_E)}$ 

\begin{align*}
 H_\mu(P^F)& \leq H_{\mu}(P^F|F)+H_{\mu}(F), \\
& \leq\sum_E \mu(F_E)H_{\mu_{F_E}}(P^E)+H_{\mu}(F).
\end{align*}
By applying Lemma \ref{lem:Miss} to each $E$ and $\mu_{F_E}$ we get for any $m$:
\begin{align}\label{eins}
\int-\log \mu\left(P^{F(x)}(x)\right) \, d\mu(x)& \leq\sum_E  \mu(F_E) \sharp E\frac{H_{\mu_{F_E}^{E}}(P^m)}{m}+3m\mu(F_E)\sharp \partial E \log \left(\sharp P\right)+H_{\mu}(F).
\end{align}
Observe now that $\left( \int \sharp F(x)\, d\mu(x)\right)\, \mu^F=\sum_E\left(\mu(F_E) \sharp E \right)\, \mu_{F_E}^{E}$ so that we obtain by concavity of the static entropy in the measure 
\begin{align}\label{zwei} \sum_E  (\mu(F_E) \sharp E) \, H_{\mu_{F_E}^{E}}(P^m)\leq \left( \int \sharp F(x)\, d\mu(x)\right) H_{\mu^F}(P^m).
\end{align}
One easily concludes the  proof by combining the above inequalities (\ref{eins}) and (\ref{zwei}). 
\end{proof}

\section{$\phi$-hyperbolic empirical measures}
We consider in this section general topological systems $(X,T)$ with a continuous observable $\phi:X\rightarrow \mathbb R$.  We introduce different notions of hyperbolic times with respect to $\phi$ at  $x\in X$,  whose associated asymptotic densities  are related with  the  ergodic components $\nu$ of the  limit empirical measures $\mu=\lim_n\mu_x^n$ with $\int \phi \, d\nu >0$  or $\int \phi\,d\nu\geq 0$ (see Lemma \ref{compa} and Lemma \ref{equiv} below).\\

  For a subset $E$ of $\mathbb N$ and $M\in \mathbb N$ we let  $E\langle M \rangle=\{k\in \mathbb N, \ \exists l, m \in E \text{ with } l\leq k<m \text{ and } m-l\leq M\}$.  The frequency of $E$ in $[1,n]$ is denoted by 
  $d_n(E)=\frac{\sharp E\cap [1,n]}{n}$. Then we consider the usual  upper and lower asymptotic density,  $\overline{d}(E):=\limsup_n d_n(E)$ and $\underline{d}(E)=\liminf_n d_n(E)$.  Observe that $E\langle M \rangle \subset E(M)$ and  one easily checks  for $N\geq M$
  \begin{equation}\label{eq:angle}
  \forall n\geq 1, \  d_n\left(E(M)\setminus E\langle N \rangle\right)\leq \frac{M}{n}+\frac{M}{N}.
   \end{equation}
   A connected component of $E$ is a maximal interval of integers contained in $E$. 
\subsection{Hyperbolic times}
We first recall the standard notion of hyperbolic times introduced in the field of smooth dynamical systems in the works  of J. Alves \cite{alv, ABV}.  
 Let $\delta>0$ and $a=(a_n)_{\mathbb N \ni n< N}$, $\mathrm N\in \mathbb N\cup \{\infty\}$ be a sequence of  finite\footnote{In this case we index the sequence with the $\mathrm N$ first non negative integers, but we may similarly consider sequences indexed by any interval of integers in $\mathbb N$. } or infinite real numbers. A positive  integer $p<\mathrm N$ is said to be a \textbf{$\delta$-hyperbolic  time }w.r.t. $a=(a_n)_n$  when $\sum_{k\leq l < p}a_l\geq (p-k)\delta$ for all $k=0,\cdots, p-1$.    We let $E_{a}^\delta$  be  the set of $\delta$-hyperbolic times w.r.t. $a=(a_n)_n$.  \\
 
 We define now a weaker notion.  For $M\in \mathbb N^*$,  the integer $p<\mathrm N$ is said to be a \textbf{$(\delta,M)$-weakly hyperbolic  time} w.r.t. $(a_n)_{n\leq N}$  when  $\sum_{k\leq l < p}a_l\geq (p-k)\delta$ for $k=\max(p-M,0),\cdots,  p-1$.
  We denote by $F_{a}^{\delta,M}$    the set of $(\delta,M)$-weakly hyperbolic times w.r.t. $a=(a_n)_n$.  Clearly $E_{a}^\delta=\bigcap_{M}F_{a}^{\delta,M}$. \\
  
   The set of $(\delta,M)$-weakly hyperbolic  times is a priori not nondecreasing in $M$. To overcome this difficulty we will work with the $\delta/2$-hyperbolic times of the connected components of  $F_{a}^{\delta,M}(M)$.  More precisely we introduce the set $G_a^{\delta,M}$ of  \textbf{$(\delta,M)$-midly hyperbolic times} defined as follows.  For a subset $E$ of $\mathbb N$ we write  $a_E:= 
 (a_k)_{E\ni k<\mathrm N}$. Then we put   $G_a^{\delta,M}:=\bigcup_{I}E_{a_I}^{\delta/2}$, 
   where $I$ runs over the connected components of $F_{a}^{\delta,M}(M)$.  Observe that for such an interval of integers $I$, we have $\sum_{k\in I}a_k\geq\delta\sharp I$. In particular if $\mathrm N=+\infty$ and  $\|a\|_{\infty }=\sup_k |a_k|<+\infty$ then  it follows from  Pliss Lemma (see e.g. \cite{ABV}) that for some $\alpha>0$  and for $M$ large enough both depending only on $\delta$ and $\|a\|_{\infty }$, we have \begin{equation}\label{Pliss}\overline{d}\left(G_a^{\delta,M}\right)\geq \alpha \overline{d}\left(F_{a}^{\delta,M}(M)\right).\end{equation}
We also let    $G_a^{\delta,M}\left(\left(N\right)\right)=G_a^{\delta,M}(N)\cap F_{a}^{\delta,M}(M)$ for any $M,N\in \mathbb N^*$.
   \\

In the next subsections   we consider  a general  invertible topological system $(X,T)$ together with a continuous  real observable $\phi:X\rightarrow \mathbb R$.  For $x\in X$ we denote the sequence $\left(\phi(T^nx)\right)_{n\in \mathbb N}$ by $\Phi_x$.   
 Moreover for any interval of integers $I$  we let $\Phi_x^I$ be the finite sequence $(\phi(T^kx))_{k\in I}$.
We  also consider the following Birkhoff sums w.r.t. $\phi$ at  $x\in X$ :
$$\forall n \in \mathbb N, \ \phi_n(x):=\sum_{0\leq k<n}\phi(T^kx),$$
$$\overline{\phi_*}(x):=\limsup_n\frac{\phi_n(x)}{n}.$$ 
When the limit exists we just write $\phi_*(x)=\overline{\phi_*}(x)$. This is the case for  typical $\mu$-points $x$ of any $T$-invariant measure $\mu$ by Birkhoff's ergodic theorem.  Moreover, when $\mu$ is ergodic  we have $\phi_*(x)=\int \phi\,d\mu$ for $\mu$-a.e. $x$.

\subsection{Hyperbolic empirical measures}
Following \cite{BCS,BCS1} we firstly consider here empirical limits with respect to the set of hyperbolic times  $E=E^\delta_{\Phi_x}$ for some $\delta>0$  and $x\in X$. 
 
\begin{lem}\label{one}Let $\delta>0$ and let $\mathfrak n\subset \mathbb N$ be a infinite subsequence  of integers.  Let $(\xi_n)_{n \in \mathfrak n}$  be a sequence of probability measures on $X$ such that 
 $\zeta_n^{M}:=\int\mu_x^{n}\left[ \overline{E^\delta_{\Phi_x}(M)}\right] \, d\xi_n(x)$  are converging,  when $n\in \mathfrak n$ goes to infinity,  to  $\zeta^{M}$  for any $M\in\mathbb N^*$.  We let $\zeta$ be the nonincreasing limit in $M$ of $(\zeta^M)_M$. 

Then we have 
$\int \phi\, d\zeta\leq \zeta(X)\delta.$

\end{lem}  
\begin{proof}
By Equation (\ref{eq:angle}),   we may replace $E^\delta_{\Phi_x}(M)$ by $E^\delta_{\Phi_x}\langle M\rangle$ in the definition of $\zeta_n^{M}$ and the limit $\zeta$ will be the same.  Now,  when $0\leq k<l$ are two consecutive times in $E^\delta_{\Phi_x}$, then $[k,l[$ is a \textit{neutral block} as defined in \cite{BCS}, i.e. $\phi_m(T^k x)< m\alpha$ for all $1\leq m < \cdots <l-k$.  From $\phi_{l-k-1}(T^k x)< (l-k-1)\delta$ and  $\phi_{l-k}(T^kx)\geq (l-k)\delta$ we get:
\begin{equation}\label{eq:connected}
|\phi_{l-k}(T^k x)-(l-k)\delta|\leq \|\phi\|_{\infty}.
\end{equation}
When $n$ belongs to $[k,l)$, we just have  
\begin{equation}\label{eq:connectedd}\phi_{n-k}(T^kx)<(n-k)\delta.
\end{equation} When $l-k>M$, then $[k,l[$ is a connected component  of $\overline{E^\delta_{\Phi_x}\langle M\rangle }$. In particular the number of such component $[k,l[$ with $[k,l[\cap [0,n[\neq \emptyset$ is less than or equal to $\frac{n}{M}+1$. 
Therefore by summing (\ref{eq:connected}) over the intervals  $[k,l[\cap [0,n[$ and (\ref{eq:connectedd}) we get 
$$\int \phi\,  d\zeta_n^{M}-\zeta_n^{M}(X)\delta\leq \|\phi\|_{\infty}\left(  \frac{1}{M}+\frac{1}{n}\right),$$ therefore by taking the limit in $n$ then in $M$ we conclude that \begin{align*}
\int \phi\, d\zeta&=\lim_{M}\lim_n\int \phi\, d\zeta_n^{M},\\
&\leq \zeta(X)\delta.
\end{align*}
\end{proof}

\begin{rem}\label{poli}When moreover   $n\in E^\delta_{\Phi_x}$ for $\xi_n$-a.e., then $\int \phi\, d\zeta=\zeta(X)\delta.$ In this case  any  connected component  $[k,l[$ of $\overline{E^\delta_{\Phi_x}\langle M\rangle} $  either lies in $[0,n)$ or are disjoint of $[0,n)$. Therefore by summing (\ref{eq:connected}) over the connected component  $[k,l[$ of $\overline{E^\delta_{\Phi_x}\langle M\rangle }$  inside $[0,n)$ we get 
$\left|\int \phi\,  d\zeta_n^{M}-\zeta_n^{M}(X)\delta\right|\leq  \frac{\|\phi\|_{\infty}}{M},$ then $\int \phi\, d\zeta=\lim_{M}\int \phi\, d\zeta^{M}=\zeta(X)\delta.$
\end{rem}

 The  lower  asymptotic density of hyperbolic times is defined as follows:
$$\underline{d}_\phi(x):=\lim^\nearrow_{\delta\rightarrow 0}\lim^{\nearrow}_{M\rightarrow +\infty}\underline{d}(E^{\delta}_{\Phi_x}(M)).$$

The next lemma,  which follows essentially from \cite{BCS}, illustrates how full density of hyperbolic times at $x$ is reflected on the measures in $pw(x)$.  

\begin{lem}\label{compa}\cite{BCS}
The following properties are equivalent :
\begin{enumerate}
\item $\underline{d}_\phi(x)=1$,
\item for any $\mu\in pw(x)$ we have $\phi_*>0$ $ \mu$-a.e.
\end{enumerate}
\end{lem}
\begin{proof}[Sketch of Proof]
Let $\mu\in pw(x)$ and let $\mathfrak n$   be a subsequence of positive integers with $\mu=\lim_{\mathfrak n\ni n \rightarrow +\infty}\mu_x^n$.  By a Cantor diagonal argument we may assume $\left(\mu_x^{n}[E^\delta_{\Phi_x}(M)] \right)_{n\in \mathfrak n}$ is converging for any $M\in \mathbb N^*$ and $\delta\in \mathbb Q$ to some $\mu_{M,\delta}$.  The measures $\mu_{M,\delta}$ are nondecreasing in $M$ and $\delta$, when    $M$ goes to infinity and $\delta$ to $0$.  The limit $\nu$  is a $T$-invariant component of $\mu$ with $\nu(X)\geq \lim_\delta\lim_M\lim_{n\in \mathfrak n}d_n\left(E^{\delta}_{\Phi_x}(M)\right)\geq \underline{d}_\phi(x)$. Moreover 
$\phi_*>0$ $\nu$-a.e.  (see e.g. \cite{BCS} or Remark  \ref{rem:nouvpr} below).  By applying 
Lemma \ref{one} 
with $\xi_n=\delta_x$ for all $n$,  the limit $\zeta$   is just the difference $\mu-\nu$, therefore $\int \phi \,d(\mu-\nu)\leq 0$. 

We explain now the equivalence of $(1)$ and $(2)$.
\begin{itemize}
\item $(1)\Rightarrow (2)$: If $\underline{d}_\phi(x)=1$, then the component $\nu$ of $\mu$ is a probability, therefore $\nu=\mu$ and $\phi_*>0$ $\mu$-a.e. 
\item $(2)\Rightarrow (1)$: there are sequences $(\delta_k)_k$, $(M_k)_k$ and $\mathfrak n=(\mathfrak n_k)_k$ such that 
\begin{itemize}
\item $(\delta_k)_k$ is decreasing to zero,
\item $(M_k)_k$ and $\mathfrak n=(\mathfrak n_k)_k$ are integer valued sequences going increasingly,
\item $\lim_k d_{\mathfrak n_k}\left(E^{\delta_k}_{\Phi_x}(M_k)\right)=\underline{d}_\phi(x).$
\end{itemize}
For any $\delta>0$ and for any $M\in \mathbb N^*$ we have 
$$\limsup_{n\in \mathfrak n}d_n\left(E^{\delta}_{\Phi_x}(M)\right) \leq \lim_k d_{\mathfrak n_k}\left(E^{\delta_k}_{\Phi_x}(M_k)\right))=\underline{d}_\phi(x).$$

By taking a subsequence we may assume that  $\mu_x^n$ is converging to $\mu \in pw(x)$ when $n\in \mathfrak n$ goes to infinity.  By hypothesis (2) we have $\phi_*(y)>0$ for  $\mu$-a.e.  $y$.  As $\mu-\nu$ is a component of $\mu$ with $\int \phi\, d(\mu-\nu)\leq 0$ we have necessarily $\nu=\mu$. In particular $1=\nu(X)=\lim_{\delta\to 0}\lim_{M\to +\infty}\lim_{n\in \mathfrak n}d_n\left(E^{\delta}_{\Phi_x}(M)\right)=\underline{d}_\phi(x)$.
\end{itemize}

\end{proof}

 We say  $\underline{d}^\phi=1$ uniformly on $\mathrm E\subset X$ when $\underline{d}^\phi(x)=1$ for all $x\in \mathrm E$ and the limits in 
$M$ and $\delta$ and the liminf in $n$ are uniform in $x\in \mathrm E$, i.e. for all $\epsilon>0$ there is $\delta_0$, $M_0$ and $n_0$ such that for any $\delta<\delta_0$, $M>M_0$ and $n>n_0$ we have 
$$\forall x\in \mathrm E, \ d_n\left(E_{\Phi_x}^{\delta}(M)\right)>1-\epsilon.$$
By Egorov theorem, if  $\lambda$ is some Borel probability measure on $X$ (e.g. the Lebesgue measure for a compact smooth manifold $X$ as in the next sections) and    $\underline{d}^\phi=1$ on a subset $\mathrm F$ of $X$,  then there is $\mathrm E\subset \mathrm F$ with $\lambda(\mathrm E)$ arbitrarily close to $\lambda(\mathrm F)$ such that  $\underline{d}^\phi=1$ uniformly on $\mathrm E$.

\subsection{Weakly hyperbolic empirical measures}
We deal now with empirical measures associated to weakly hyperbolic times,  which is 
the main new tool used in the present paper.

\begin{lem}\label{nouv}Let $\delta>0$ and let $\mathfrak n_q\subset \mathbb N$, $q\in \mathbb N$, be  infinite subsequences of integers.  Let $(\xi_n^q)_{q\in \mathbb N, \ n\in \mathfrak n_q}$  be probability measures such that 
$\nu_n^{M,q}:=\int\mu_x^{n}[F^{\delta,M}_{\Phi_x}(M)] \, d\xi_n^q(x)$ are converging,  when $n\in \mathfrak n_q$ goes to infinity,  to $\nu^{M,q}$  for any $M,q$.  

Then for any limit $\nu$ of the form $\nu=\lim_M\lim_{q\in \mathfrak q}\nu^{M,q}$ for some infinite subsequence $\mathfrak q$, we have  $\phi_*(x)\geq \delta$  for $\nu$ a.e.  $x$.
\end{lem} 

\begin{proof}
Let $K_M:=\{x\in X, \ \phi_m(x)\geq m\delta \text{ for some } M\geq m\geq 1\}$  and 
$K=\bigcup^\nearrow_{M\geq 1}K_M$.   For $n\in \mathfrak n_q$ we also let $\eta_{n}
^{M,q}:=\int\mu_x^{n}[H_{\Phi_x}^{\delta,M}(M)]\, d\xi_n^q(x)$  with $H_{\Phi_x}^{\delta,M}
(M)=\bigcup_{N\leq M}F_{\Phi_x}^{\delta,N}(N)$.  By using a Cantor diagonal argument we may assume the corresponding limits  
$\eta^{M,q}$, $M,q\in \mathbb N$,  exist when $n\in \mathfrak n_q$ goes to infinity.  Note that $\nu^{M,q}$ is a component of $\eta^{M,q}$.  As 
$K_M$ is a compact set of $X$,  we have $\eta^{M,q}(K_M)\geq \lim_n\eta_n^{M,q }(K_M)$, but 
it follows from the definition of $K_M$ that $\eta_n^{M ,q}(K_M)=\eta_n^{M,q }(X)$, therefore $
\eta^{M,q}(K_M) \geq  \lim_n\eta_n^{M,q }(K_M)=\eta^{M,q}(X)$.  Then  by replacing $\mathfrak q $ by a subsequence  we may assume the limit $\eta=\lim^{\nearrow}_M\lim_{q\in \mathfrak q } \eta^{M,q}$ is well defined. The measure $\nu$ is a component of $\eta$. Moreover we have  \begin{align*}
\eta(K)&=\lim^{\nearrow}_{N}\eta(K_N),\\
&=\lim^{\nearrow}_{N}\lim^{\nearrow}_{M}\lim_q\eta^{M,q}(K_N),\\
&\geq \lim^{\nearrow}_{N}\lim_q\eta^{N,q}(K_N),\\
&\geq \lim^{\nearrow}_{N}\lim_q\eta^{N,q}(X),\\
&\geq \eta(X).
\end{align*} In particular $\nu(K)=\nu(X)$.  
 As $\nu$ is invariant (a priori $\eta$ is not) we get $\nu\left(\bigcap_{p\in \mathbb N}T^{-p}K \right)=\nu(X)$ and one checks easily that $\overline{\phi_*}(x)\geq \delta$ for $x\in \bigcap_{p\in \mathbb N}T^{-p}K$.\end{proof}

\begin{rem}\label{rem:nouvpr}
With the notations of Lemma \ref{one} and Lemma \ref{nouv} we assume the limit $\mu=\lim_{n\in \mathfrak n}\int\mu_x^{n} \, d\xi_n(x)$ exists. 
Then, as $E^{\delta}_{\Phi_x}$ is a subset of  $F^{\delta,M}_{\Phi_x}$,  the measure $\overline{\zeta}:=\mu-\zeta$ with $\zeta$  as in Lemma \ref{one} is a component of the measure $\nu$ given in Lemma \ref{nouv} with $\xi_n^q=\xi_n$ for all $q$.  In particular  $\phi_*(x)\geq \delta$  for $\overline{\zeta}$ a.e.  $x$.
\end{rem}

 Contrarily to $\left(E^{\delta}_{\Phi_x}(M)\right)_M$ the sequence $\left(F_{\Phi_x}^{\delta,M}(M)\right)_M$ is a priori neither nonincreasing or nondecreasing.  Therefore we define two versions of the associated upper asymptotic density:

$$\overline{d}^+_\phi(x):=\lim^{\nearrow}_{\delta\rightarrow 0}\limsup_{M\rightarrow +\infty} \overline{d}(F_{\Phi_x}^{\delta,M}(M)),$$
$$\overline{d}^-_\phi(x):=\lim^{\nearrow}_{\delta\rightarrow 0}\liminf_{M \rightarrow +\infty} \overline{d}(F_{\Phi_x}^{\delta,M}(M)).$$

We say $\overline{d}^+_\phi=0$   uniformly on a subset $\mathrm E$ of $X$ when $\overline{d}^+_\phi(x)=0$ for all $x\in \mathrm E$ and the limsup in $M$ and $n$ defining $\limsup_M \overline{d}(F_{\Phi_x}^{\delta,M}(M))=0$ are uniform in $x\in \mathrm E$ for all $\delta>0$,  i.e.  
$$\forall \delta>0 \ \forall \epsilon>0\  \exists M_0\ \forall M>M_0 \ \exists n_0 \ \forall n>n_0 \ \forall x\in \mathrm E, \ d_n\left(F_{\Phi_x}^{\delta,M}(M)\right)<\epsilon.$$
Again, we may apply Egorov's theorem to get sets where  $\overline{d}^+_\phi=0$   uniformly : if $\lambda$ is a probability measure on $X$ and $\mathrm F$ is a subset with $\overline{d}^+_\phi(x)=0$ for all $x\in \mathrm  F$, then  $\overline{d}^+_\phi=0$   uniformly on a subset $\mathrm E$ of $\mathrm F$ with $\lambda(\mathrm E)$ arbitrarily close to $\mathrm F$. 

\begin{lem}\label{two}Assume $\overline{d}^+_\phi=0$ uniformly on $\mathrm E \subset X$.  Let $(\xi_n)_{n}\in \mathcal M(X)^{\mathbb N}$ with $\xi_n(\mathrm E)=1$ and let  $(\nu_n)_n=\left( \int \mu_x^n\, d\xi_n(x) \right)_n$.

Then for any limit $\nu$ of $(\nu_n)_n$ we have $\phi_*(x)\leq 0$ for  $\nu$-a.e.  $x$.

\end{lem}
Before proving Lemma \ref{two} we introduce some  notations that will also be used in the last section.   For $\delta>0$, $M\in \mathbb N^*$, we  will let $H_{\delta}:=H_{\delta}(\phi)$ and  $O_M:=O_M(\phi,\delta)$ be the sets defined as follows: 
$$H_{\delta}:=\{x, \ \forall m>0 \  \phi_m(T^{-m}x)>m\delta \},$$
$$O_M=\bigcup_{k=1,\cdots, M}T^{-k}\left\{x, \  \  \phi_m(T^{-m}x)> m\delta \text{ for all } m=1,\cdots M\right\}.$$

\begin{proof}

Assume $\nu(\phi_*>0)=\lambda>0$.  Fix $\delta>0$ small enough such that 
$\nu(\phi_*>\delta)=\beta >\lambda/2$.  For $M$ large enough we have 
 $\limsup_n \sup_{x\in \mathrm E}d_n\left(F_{\Phi_x}^{\delta,M}(M)\right)<\lambda/4$.

 Let  $K=\bigcup_{k\in \mathbb N^*} T^{-k}H_\delta$. By the ergodic maximal inequality, if $\eta$ is an ergodic component of  $\nu$ with $\phi_*>\delta$ $\eta$-a.e.  then $\eta(H_\delta)>0$ and $\eta(K)=1$ by ergodicity. Therefore  we have $\nu(K)\geq\beta>0$.    Let $\mathfrak n\subset \mathbb N$ be an infinite  subsequence with 
$\lim_{n\in \mathfrak n}\nu_n=\nu$.  For all $M$ the set $O_M$ is an open neighborhood of $\bigcup_{k=1,\cdots, M} T^{-k}H_\delta$, therefore for $M$ large enough we get 
$$\liminf_{n\in \mathfrak n}\nu_n(O_M)\geq \nu(O_M)\geq \nu\left(\bigcup_{k=1,\cdots, M} T^{-k}H_\delta\right)\geq \beta/2.$$ Observe also that when $T^lx$, $l\in \mathbb N$,   lies in $O_M$ then $l$ belongs to  $F_{\Phi_x}^{\delta,M}(M)$.  Thus we get the following contradiction for $M$ large enough:
 \begin{align*}
 \limsup_{n\in \mathfrak n}\nu_n(O_M)&=  \limsup_{n\in \mathfrak n} \int \mu_x^n(O_M)\,d\xi_n(x),\\
 &\leq  \limsup_n \sup_{x\in \mathrm E}d_n\left(F_{\Phi_x}^{\delta,M}(M)\right),\\
 &<\lambda/4<\beta/2.
\end{align*}
\end{proof}

\begin{rem}\label{susu}When $\xi_n$ is just the Dirac measure at some $x$ for any $n$   with $\overline{d}^-_\phi(x)=0$, i.e. $\nu_n=\mu_x^n$ the same conclusion holds, i.e. if $\overline{d}^-_\phi(x)=0$, then for any $\nu\in pw(x)$ we have $\phi_*(x)\leq 0$ for $\nu$-a.e. $x$ .  Indeed, to get the contradiction at the end of the above proof, one only needs to consider some (not any) large $M$ with $\overline{d}(F_{\Phi_x}^{\delta,M}(M))$ small. 
\end{rem}

Zero upper density of weakly-hyperbolic times at $x$  is also related  with the non-negativity of $\phi_*$ on $pw(x)$,  as stated in the following Lemma which may be compared with Lemma \ref{compa}.

\begin{lem}\label{equiv}
The following properties are equivalent :
\begin{enumerate}
\item  $\overline{d}^+_\phi(x)=0$,
\item  $\overline{d}^-_\phi(x)=0$,
\item for any $\mu\in pw(x)$ we have $\phi_*(x)\leq  0$ for $\mu$-a.e.  $x$.
\end{enumerate}
\end{lem}

\begin{proof}
Clearly $ \overline{d}^+_\phi(x)\geq \overline{d}^-_\phi(x)$ so that $(1)  \Rightarrow (2).$ 
Then $(2) \Rightarrow (3)$   follows from Remark \ref{susu}. Therefore it is enough to show $(3) \Rightarrow (1)$.  We argue by contradiction.  Assume $\lim^{\nearrow}_{\delta}\limsup_M \overline{d}(F_{\Phi_x}^{\delta,M}(M))>0$ and let us show there is $\mu\in pw(x)$ with $\mu(\phi_*>0)>0$. Fix $\delta>0$ with $\limsup_M \overline{d}(F_{\Phi_x}^{\delta,M}(M))=\lambda>0$.  For infinitely many  $M$ there is a sequence $\mathfrak n_M$ such that $\mu_x^n[F_{\Phi_x}^{\delta,M}(M)]$  is converging to $\nu^M$ with   $\nu^M(X)=\overline{d}(F_{\Phi_x}^{\delta,M}(M))>\lambda/2$, when $n\in \mathfrak n_M$ goes to infinity.   Then $\int \phi \, d\nu^M= \lim_{n \in \mathfrak n_M}\int \phi\,d\mu_x^n[F_{\Phi_x}^{\delta,M}(M)]\geq  \delta \lambda/2 $. We may also assume that $\mu_x^n$ is converging to  some $\mu_M\in pw(x)$, when $n\in \mathfrak n_M$ goes to infinity.   
Let $(\mu, \nu)$ be a weak-$*$ limit of  
the sequence $(\mu_M,\nu_M)_M\in \left( \mathcal M(X)\times \mathcal M(X)\right)^2$. Then $\mu\in pw(x)$ has $\nu\neq 0$ as a $T$-invariant   component with $\int \phi\, d\nu=\int \phi_*\, d\nu>0$. 
Therefore $\mu(\phi_*>0)>0$.  This contradicts (3).
\end{proof}
\subsection{Midly hyperbolic  empirical measures}
We consider midly empirical measures,  i.e. empirical measures with respect to $G_{\Phi_x}^{\delta,M}(N)$ for $1\ll N\ll M$.   The next lemma (and its proof) is analogous to Lemma \ref{nouv} for weakly hyperbolic empirical measures.  In this previous lemma, the measures $\xi_n^q$ are not allowed to depend on $M$ because of the lack of monotonicity in $M$ of $F^{\delta,M}_{\Phi_x}$.  This difficulty may be overcome for midly hyperbolic  empirical measures as follows.

\begin{lem}\label{nouvv}Let $\delta>0$ and $(\xi_n^M)_{n,M}\in \mathcal M(X)^{\mathbb N^2}$.  Let $\mathfrak n_M\subset \mathbb N$, $M\in \mathbb N$, be infinite subsequences such that 
$\nu_n^{M,N}:=\int\mu_x^{n}[G^{\delta,M}_{\Phi_x}\left(\left(N\right)\right)] \, d\xi_n^M(x)$ are converging for any $M,N$,  when $n\in \mathfrak n_M$ goes to infinity,  to $\nu^{M,N}$.   Let $\mathfrak M\subset \mathbb N$ be an infinite subsequence such  the limits $\nu^N=\lim_{M\in \mathfrak M}\nu^{M,N}$ exist for all $N$.    

Then for the nondecreasing limit $\nu$ of $(\nu_N)_N$ we have  $\phi_*(x)\geq \frac{\delta}{2}$  for $\nu$ a.e.  $x$.
\end{lem} 

\begin{proof}
Denote $K_N:=\{x\in X, \ \phi_m(x)\geq m\delta/2 \text{ for some } N\geq m\geq 1\}$  and 
$K=\bigcup_{N\geq 1}K_N$.    As 
$K_N$ is a compact set of $X$,  we have $\nu^N(K_N)\geq \lim_M\lim_{n\in \mathfrak n_M}\nu_n^{M,N}(K_M)$,  but 
it follows from the definition  of $G^{\delta,M}_{\Phi_x}\left(\left(N\right)\right)$ that $\nu_n^{M,N}(K_N)=\nu_n^{M,N }(X)$, therefore $
\nu^N(K_N) \geq  \lim_M \lim_n\nu_n^{M,N }(K_N)=\nu^N(X)$.  Then $\nu(K)=\nu(X)$  
because the sequence $\nu^N$ is nondecreasing in $N$, thus we have  $\nu(K)=\nu(X)$. Then one may conclude as in Lemma \ref{nouv}.\end{proof}

\section{Empirical measures   for partially hyperbolic systems }

Let   $\Lambda$ be an attracting  set of a $\mathcal C^{1+}$ diffeomorphism $f$, i.e. $\Lambda$ is a compact invariant set with an open neighborhood $U$ satisfying $f(\overline{U})\subset U $ and $\Lambda=\bigcap_{n \in \mathbb N}f^n U$, with  a partially hyperbolic splitting $T\mathbf M|_{\Lambda} = E^u \oplus_{\succ}  E_{1} \oplus_{\succ}\cdots \oplus_{\succ} E_{k} \oplus_{\succ} E^s$ with $\mathrm{dim}(E_{i})=1$ for $i=1,\cdots, k$.  The bundles in the splitting are $f$-invariant and H\o lder continuous. We may choose a norm adapted to the splitting \cite{Gou} :   the bundles $E^u$ and $E^s$ are respectively uniformly expanding and contracting, i.e.  
$\|Df|_{E^s(x)}\|<1$ and $\|Df^{-1}f|_{E^u(x)}<1$ for all $x\in \Lambda$ and $E \oplus_{\succ}F $ means that for any unit vectors $v_E\in E(x)$ and $v_F\in F(x)$ we  have 
$\|D_xf(v_E)\|<\|D_x f(v_F)\|$.  We may assume  that the bundles and the splittings hold on the neighborhood $U$. \\
 
For any $i=1,\cdots, k$ and for any $x\in \mathbf M$ we let $\phi^i(x)=\log\|D_xf|_{E_{i}}\|$.  We also put $\phi^0(x)=-\log\|D_{fx}f^{-1}|_{E^u}\|$ and $\phi^{k+1}(x)=\log\|D_xf|_{E^s}\|$.  The $i^{th}$ center exponent of an invariant measure $\mu$ is then 
$$\forall i=1,\cdots, k, \ \phi_i(\mu)=\int \phi^i_*\, d\mu.$$
By the domination property there is $a>0$ such that $(\phi_{i+1}-\phi_i)(\mu)>a$ for any invariant probability measure $\mu$ and for any $i=1,\cdots, k-1$.  An invariant probability measure $\mu$ with $\mu(\phi_*^i=0)=0$ for all $i$ is said to be hyperbolic.  \\

To simplify the notations we let for all $i=1,\cdots, k$ and  for all $x\in M$ 
$$\alpha_i(x)=\underline{d}_{\phi^i}(x)$$
and 
$$ \beta_i(x)=\overline{d}^+_{-\phi^i}(x).$$
We let also by convention $\alpha_0=1$,  $\alpha_{k+1}=0$ and  $\beta_{k+1}=0$.  We have 
$\left\{\beta_{i}>0\right\}\subset \left\{\alpha_{i}<1\right\}$. 
Indeed by Lemma \ref{equiv}, when  $\beta_{i}(x)>0$,  there is $\mu\in pw(x)$ with $\mu\left(\phi^{i}_*<0\right)>0$.  In particular  from Lemma \ref{compa} we get   $\alpha_{i}(x)<1$.   Observe also that when  $\beta_{i}(x)=0$ and $\alpha_i(x)<1$, then there is $\mu\in pw(x)$ with $\mu(\phi^{i}_*=0)>0$.  In the definition of $\beta_i$ we could also have chosen  $\beta_i=\overline{d}^-_{-\phi^i}$ without making any difference in the next statements because of the equivalence $(1)\Leftrightarrow (2)$ in Lemma \ref{equiv}.

\begin{prop}\label{fond}For any $i=0,\cdots, k$ and for Lebesgue a.e.  point $x\in U$ with $\alpha_{i+1}<\alpha_i(x)=1$, we have 
\begin{itemize}
\item either $\beta_{i+1}(x)>0$ then $x$ lies in the basin of an ergodic hyperbolic SRB measure,
\item or $\beta_{i+1}(x)=0$, then some   $\mu\in pw(x)$  admits a non-hyperbolic SRB component $\nu$ with $\phi_*^{i+1}(x)=0$  for $\nu$-a.e.  $x$.
\end{itemize}
\end{prop}

Theorem  1 follows straightforwardly from Proposition \ref{fond} as we have 
$$U\subset\bigcup_{i=0,\cdots, k}\left\{ \alpha_i=1 \text{ and } \alpha_{i+1}<1  \right\}. $$

\begin{prop}\label{prop:second}
Let $(\mu_q)_q$ be a sequence  of  ergodic SRB measures converging to some $\mu\in \mathcal M(\Lambda,f)$. Let  $i\in \{0,  \cdots,  k\}$ be the  (unique) integer with $\mu(\phi^i_*>0)=1$ and  $\mu(\phi^{i+1}_*>0)<1$.  We have:
\begin{itemize}
\item either $\mu(\phi^{i+1}_*< 0)>0$, then $\mu_q$ is equal to $\mu$ for large $q$  and $\mu$ is an ergodic  hyperbolic SRB measure with unstable index $i$,  i.e.  $\phi^{i+1}_*(x)< 0<\phi^i_*(x)$ for  $\mu$-a.e $x$,
\item or $\mu(\phi^{i+1}_*\geq 0)=1$, then $\mu$ has a non-hyperbolic SRB component $\nu$ with $\phi^{i+1}_*(x)=0$ for $\nu$-a.e. $x$.  
\end{itemize}
\end{prop}

In particular,  any limit of distinct ergodic SRB measures has an SRB non-hyperbolic component.  When all SRB measures are assumed to be hyperbolic,  there are therefore only finitely many SRB measures. Together with Theorem \ref{mmain} one easily completes the proof of Theorem \ref{ttwo}.\\

\begin{ques} For $i=1,\cdots,  k$ we let $G_i$ be the bundle $G_i=E^u\oplus 
  E_1\oplus\cdots \oplus E_i$.  An $f$-invariant   measure $\mu$ is called an $i$-Gibbs state when $\phi^i_*>0$ $\mu$-a.e.   and the Ledrappier-Young entropy of $\mu$ with respect to the Pesin $i$-unstable manifolds  $W^{i}$ tangent to $G_i$  is equal to the sum of  the $\mathrm{dim}(G_i)$ first positive exponents of $\mu$.  Equivalently the conditional measures along $W^{i}$ are absolutely continuous w.r.t. the Lebesgue measure on $W^{i}$.  We believe that for    Lebesgue a.e.  $x$ with $\alpha_i(x)=1$ any $\mu\in pw(x)$ is a $i$-Gibbs state.  However our proof does not allow us to prove it.  Also, in Theorem \ref{mmain} can one straighten the second item by showing that \textbf{any} empirical measure $\mu\in pw(x)$ has an SRB non-hyperbolic  component?

\end{ques}

\subsection{Gibbs property for empirical measures at hyperbolic times}

 We let $\psi_i(x):= \mathrm{Jac}(Df|_{G_i})(x)$ for any $x\in U$ and any $i=1,\cdots,  k$.  For a finite set of nonnegative  integers $F$ 
 we let $\psi_i^F(x)=\prod_{k\in F}\psi_i(f^kx)$. Also we write $\partial F$ for the symmetric difference $F\Delta (F+1)$. Let $\mathfrak C_i$ 
 be an invariant cone around $G_i$. A smooth embedded disc $D$ 
 is said tangent to $\mathfrak C_i$ when the dimension of $D$ is
  equal to the dimension of $G_i$ and the tangent space of $D$ at 
  any $x\in D$ is contained in $\mathfrak C_i(x)$. The next statement is borrowed from  \cite{bd}.   We recall that for a partition $P$  of $\mathbf M$ the set $P^{F}(x)$ denotes the atom of the iterated partition $P^{F}=\bigvee_{k\in F}f^{-k}P$ containing the point $x$.

 \begin{lem}\cite[Proposition 2.7]{bd}\label{dist}For any  disc $D$ tangent to $\mathfrak C_i$, for any $\delta>0$ and  for any $\epsilon>0$ there are $C,C'$ and $\alpha>0$  such that we have  for any partition $P$ with diameter less than $\alpha$, for any $x\in U$, for any $ n\in \mathbb N$ and for any set of integers $F_n\subset [0,n[$ with $\partial F_n \subset  E^{\delta}_{\Phi^x_i} $: 
\begin{equation}\label{eq:gibbs} \mathrm{Leb}_D(P^{F_n}(x))\leq \frac{C' C^{\partial F_n}e^{\epsilon n}}{ \psi_i^{F_n}(x)}.
\end{equation}

 \end{lem}
 For a smooth embedded disc $D$ we let $d_D$ the induced Riemannian distance on $D$.  The proof of Lemma \ref{dist} is based on the following bounded geometric property at hyperbolic times. 
 
\begin{lem} \cite[Lemma 4.2]{AlP08}\label{Lem:Pliss-iterate}
For any $\delta>0$ there is $\gamma>0$ and   $N\in \mathbb N$  such that for any disk $D\subset U$    of radius $\gamma$ tangent to $\mathfrak C_i$,  for any $x\in D$ with $d_D(x, \partial D)\geq \frac{\gamma}{2}$,  for any  $ E^\delta_{\Phi_x^i}\ni n >N$,  the image $f^n(D)$ contains a disk $D_{f^nx}$ centered at $f^n x$ with radius $\gamma$  such that the diameter of $f^{-i}(D_{f^nx})$ decays exponentially fast in $i\in \{0,\cdots,n\}$.
\end{lem}


\subsection{Dynamical density on hyperbolic times}

Let $D$ be a disc tangent to $\mathfrak C_i$. We consider a subset $\mathcal D$ of $D$ such that $\overline{d}\left(E^\delta_{\Phi_i^x}\right)>0$ for any $x\in \mathcal D$. Then 
$x\in D$ is said to be a \emph{dynamical density point on $\delta$-hyperbolic times of  $\mathcal D$ with respect to $D$}
when
$$\lim_{n\to\infty,~n\in E^\delta_{\Phi^i_x}}\frac{{\rm Leb}_D\left(f^{-n}D_{f^nx}\cap \mathcal D\right)}{{\rm Leb}_D\left( f^{-n}D_{f^nx}\right)}=1,$$

where $D_{f^nx}$  is the disc of radius $\gamma=\gamma(\delta)$ at $f^nx$ given by Lemma \ref{Lem:Pliss-iterate}. We will use the following statement proved in \cite{bd}:

\begin{prop}  \cite[Theorem 3.1]{bd}\label{Thm:density-point}
With the above notations, ${\rm Leb}_D$-a.e.   $x \in \mathcal D$  is a dynamical density point of $\mathcal D$ with respect to $D$.
\end{prop}

\section{Proof of Proposition \ref{fond}}
As already mentionned   the proof of Theorem \ref{mmain} is reduced to the proof of  Proposition \ref{fond}.   We   follow the variational approach used in \cite{Bur, bd}.  In these last works
the assumptions ensure the existence of a set with positive Lebesgue measure on a smooth $l$-disc such that the set of  \textit{geometric} times (as defined in \cite{Bur}) has positive upper density. Then one may build an empirical measure  $\mu$ by pushing this disc   around these times such that the (invariant) measure $\mu$ has $l$ positive Lyapunov exponents and its entropy is larger than or equal to the sum of these $l$ exponents. Then it follows easily from the contexts in \cite{Bur,bd} that $\mu$ has exactly $l$ positive exponents,  therefore $\mu$ satisfies Pesin entropy formula.  \\

 Here the method of building SRB measure is slightly different and may be roughly resumed as follows.  We divide the topological  basin of attraction into the level sets  $\mathcal L_i=\{ \alpha_i=1 >\alpha_{i+1}\}$,  $i=1,\cdots, k$.  Then  geometric times  w.r.t.  a  disc tangent to the cone $\mathfrak C_i$ have full density at points in $\mathcal L_i$, because,  in our settings,    these geometric times coincide with the hyperbolic times for $\phi_i$.    
In this way the associated empirical  measures have entropy larger than the sum of the  $\mathrm{dim}(G_i)$ first exponents,  but these measures may have other positive exponents in general.  However  hyperbolic times w.r.t.  $\phi_{i+1}$ have not full density
on $\mathcal L_i\subset  \{1 >\alpha_{i+1}\}$. Therefore we may find  an empirical measure  with nonpositive $(i+1)^{th}$ center positive exponent, which is consequently an SRB measure.  In other terms we do not choose the empirical measures to satisfy the appropriate volume estimate implying  the lower bound on the entropy (as in \cite{Bur,bd} where geometric times do not have a priori full density)  but rather to ensure   the absence of other positive Lyapunov exponents (whereas this property is almost automatic in \cite{Bur,bd}).

\subsection{The hyperbolic case}
We first deal with the case of  hyperbolic SRB measures, i.e.  we prove the first item of Proposition \ref{fond}.  We let $\mathcal A_i:=\left\{ \alpha_i=1 \text{ and }\beta_{i+1}>0\right\}$.

\begin{prop}\label{pro:mar}
Lebesgue a.e.  point $x\in \mathcal A_i$ lies in the basin of an ergodic hyperbolic SRB measure. 
\end{prop}

The end of this subsection is devoted to the proof of Proposition \ref{pro:mar}. For any $\delta>0$ and $\lambda>0$ we let $\mathcal B_i(\lambda, \delta)$ be the  the  subset of points $x\in \mathcal A_i$  satisfying \begin{equation}\label{eq:b}\lim_P^{\nearrow}\liminf_{M\rightarrow \infty}\overline{d}\left(G_{-\Phi_{i+1}^x}^{\delta,M}\cap E^{\delta}_{\Phi_i^x}(P) \right)>\lambda.
\end{equation}

\begin{lem}\label{lem:abo}
$$\mathcal A_i\subset\bigcup_{\lambda, \delta\in \mathbb Q^+}\mathcal B_i(\lambda, \delta).$$
\end{lem}
\begin{proof}
For any $x$ with  $\beta_{i+1}(x)>0$ we have  $\overline{d}^-_{-\phi^{i+1}}(x)>0$ by Lemma \ref{equiv}.  Therefore  there is $\delta'>0$ with $\liminf_{M\rightarrow \infty}\overline{d}\left(F_{-\Phi_{i+1}^x}^{\delta',M}(M)\right)>0$ and then  it  follows from Inequality (\ref{Pliss}) that $\liminf_{M\rightarrow \infty}\overline{d}\left(G_{-\Phi_{i+1}^x}^{\delta',M}\right)>0$. When moreover $\alpha_{i}(x)=1$, there is $\delta''>0$ such that 
$\lim_{P}^\nearrow\underline{d}\left(E^{\delta''}_{\Phi_i^x}(P)\right) >1-\frac{1}{2}\liminf_{M\rightarrow \infty}\overline{d}\left(G_{-\Phi_{i+1}^x}^{\delta',M}\right)$. Therefore by taking $\delta, \lambda\in \mathbb Q$ with $0<\delta<\min(\delta',\delta'')$ and $0<\lambda<\frac{1}{2}\liminf_{M\rightarrow \infty}\overline{d}\left(G_{-\Phi_{i+1}^x}^{\delta',M}\right)$  we get 

\begin{align*}
\overline{d}\left(G_{-\Phi_{i+1}^x}^{\delta,M}\cap E^{\delta}_{\Phi_i^x}(P) \right)&\geq\overline{d}\left(G_{-\Phi_{i+1}^x}^{\delta,M}\right) +\underline{d}\left( E^{\delta}_{\Phi_i^x}(P) \right)-1,\\
&\geq\overline{d}\left(G_{-\Phi_{i+1}^x}^{\delta',M}\right) +\underline{d}\left( E^{\delta''}_{\Phi_i^x}(P) \right)-1,\\\lim_P^{\nearrow}\liminf_{M\rightarrow \infty}\overline{d}\left(G_{-\Phi_{i+1}^x}^{\delta,M}\cap E^{\delta}_{\Phi_i^x}(P) \right)& \geq\liminf_M\overline{d}\left(G_{-\Phi_{i+1}^x}^{\delta',M}\right) +\lim^\nearrow_P\underline{d}\left(E^{\delta''}_{\Phi_i^x}(P) \right)-1,\\
&> \liminf_M\overline{d}\left(G_{-\Phi_{i+1}^x}^{\delta',M}\right)+ 1-\frac{1}{2}\liminf_{M\rightarrow \infty}\overline{d}\left(G_{-\Phi_{i+1}^x}^{\delta',M}\right),\\
&>\lambda.
\end{align*}

 The proof is complete.

\end{proof}

We prove now Proposition \ref{pro:mar}.  By  Lemma  \ref{lem:abo} one only needs to consider the subset  $\mathcal B_i(\lambda, \delta)$ for any rational numbers  $\lambda, \delta>0$. Fix  
such parameters $\lambda, \delta$.   By Egorov theorem it is enough to show that Lebesgue almost every point $x$ in  a subset $\mathcal C_i$ of  $\mathcal B_i(\lambda, \delta)$ lies in  in the basin of a hyperbolic SRB measure, where   the limit in $P$ and the liminf in $M$  in $\lim_P^{\nearrow}\liminf_{M\rightarrow \infty}\overline{d}\left(G_{-\Phi_{i+1}^x}^{\delta,M}\cap E^{\delta}_{\Phi_i^x}(P) \right)$ are uniform in $x\in \mathcal C_i$,  i.e.  there exist $M_0$ and $P_0$ such that for $M>M_0$  
$$\forall x\in \mathcal C_i, \ \overline{d}\left(G_{-\Phi_{i+1}^x}^{\delta,M}\cap E^{\delta}_{\Phi_i^x}(P_0) \right)>\lambda.$$

We argue by contradiction. Assume there is a subset $\mathcal D_i$ of $\mathcal C_i$ with positive Lebesgue measure such that any point in $\mathcal D_i$ does not lie in the basin of an ergodic hyperbolic SRB measure. 
 By a standard Fubini argument,  there exists a smooth disc $D$ tangent to $\mathfrak C_i$ with $\Leb_{D}(\mathcal D_i)>0$.  By Proposition  \ref{Thm:density-point} there is a subset $\mathcal E_i$ of $D\cap \mathcal D_i$ with $\Leb_{D}(\mathcal E_i)>0$ such that any $x\in \mathcal E_i$ is a  Lebesgue density  point  for $\Leb_{D}$ of $\mathcal D_i$ at  $\delta$-hyperbolic times, i.e.  $$\lim_{n\to\infty,~n\in E^\delta_{\Phi_i^x}}\frac{{\rm Leb}_D\left(f^{-n}D_{f^nx}\cap \mathcal D_i \right)}{{\rm Leb}_D\left(f^{-n}D_{f^nx}\right)}=1.$$
    By Borel-Cantelli Lemma, for any $M>M_0$  there are  an infinite sequence $\mathfrak n_M$  and Borel subsets $A_n^M\subset \mathcal E_i$,  $n\in \mathfrak n_M$, with $\Leb_{D}(A_n^M)\geq \frac{1}{n^2}$ such that 
$$\forall n\in \mathfrak n_M\ \forall y\in A_n^M, \ d_n\left(G_{-\Phi_{i+1}^x}^{\delta,M}\cap E^{\delta}_{\Phi_i^x}(P_0) \right)>\lambda.$$

  We  consider the measures $\left(\mu_n^{M,N,P}\right)_{n\in \mathfrak n_M}$ and the associated probability measures $\left(\nu_n^{M,N,P}\right)_{n\in \mathfrak n_M}$ defined by 
$$\mu_n^{M,N,P}=\int \mu_x^n[G_{-\Phi_{i+1}^x}^{\delta,M}\left(\left(N\right)\right)\cap E^{\delta}_{\Phi_i^x}(P) ]\, d\mathrm{Leb}_{D}^{A_n^M}(x)$$
and
$$\nu_n^{M,N,P}=\frac{\mu_n^{M,N,P}}{\mu_n^{M,N,P}(\mathbf M)} $$
where $\Leb_D^{A_n^M}(\cdot)=\frac{\Leb_D(A_n^M\cap \cdot)}{\Leb_D(A_n^M)}$  is the probability measure induced by $\Leb_D$ on $A_n^M$.  

By extracting subsequences we may assume the  following successive limits  exist  $$\mu=\lim_P\lim_N\lim_M\lim_{n\in \mathfrak n_M}\mu_n^{M,N,P},$$ 
 $$\nu=\lim_P\lim_N\lim_M\lim_{n\in \mathfrak n_M}\nu_n^{M,N,P}.$$ The intermediate limits are denoted by $\mu^{M,N,P}$,  $\mu^{N,P}$,  $\mu^{P}$ and $\nu^{M,N,P}$,  $\nu^{N,P}$,  $\nu^{P}$.  Observe that $\mu\geq \lambda\nu$. 

The measures $\mu_n^{M,N,P}$ are components of $\zeta_n^{M,N}$ with 
$$\zeta_n^{M,N}=\int \mu_x^n[G_{-\Phi_{i+1}^x}^{\delta,M}\left(\left(N\right)\right) ]\, d\mathrm{Leb}_{D}^{A_n^M}(y).$$
Without loss of generality we may assume the successive limits in $n\in \mathfrak n_M$, in $M$ and in $N$ also exist for these sequences. Let $\zeta=\lim_N\lim_M\lim_{n\in \mathfrak n_M}\zeta_n^{M,N}$ be the limit measure.  The measure $\mu$ is a component of $\zeta$ and by Lemma \ref{nouvv} we have $ \phi_{i+1}^*(y)\leq -\delta/2$ for $\zeta$-a.e. $y$,  therefore for $\nu$-a.e.  $y$.  

Similarly,  since $E^{\delta}_{\Phi_i^x}\subset F_{\Phi_i^x}^{\delta, P}$   the measures $\nu_n^{M,N,P}$ are components of $\eta_n^{M,P}$ with 
$$\eta_n^{M,P}=\int \mu_x^n[ F^{\delta,P}_{\Phi_i^x}(P) ]\, d\mathrm{Leb}_{D}^{A_n^M}(y).$$
We may again assume the limits $\eta=\lim_P\lim_M \lim_{n\in \mathfrak n_M}\eta_n^{M,P}$ 
exist, so that $\nu$ is a component of $\eta$. 
 By  applying  Lemma  \ref{nouv} with  $\left(\xi_{n}^q\right)_{q,n}=\left(\mathrm{Leb}_{D}^{A_{n}^{M}}\right)_{M,n}$ we have $\phi^i_*(y)\geq \delta $ for $\eta$-a.e. $y$, therefore for $\nu$-a.e.  $y$.

Finally we check that $h(\nu)\geq \int \psi_i\, d\nu$ which will imply that $\nu$ is a hyperbolic SRB measure.  Fix $\epsilon>0$ and let $\alpha>0$ as given in Lemma \ref{dist} so that the volume estimate (\ref{eq:gibbs}) holds for the set of $\delta$-hyperbolic times $E^\delta_{\Phi^i_x}$.  Take a partition $Q$ with diameter less than $\alpha$ and with $\xi(\partial Q)=0$ for any $\xi\in \{\nu^{M,N,P},  \nu^{N,P},  \nu^{P}, \nu \ :\ M,N,P\}$.  By  applying Lemma \ref{lem:Mis} with $\mu= \mathrm{Leb}_{D}^{A_n^M}$ and $F(x)=G_{-\Phi_{i+1}^x}^{\delta,M}\left(\left(N\right)\right)\cap E^{\delta}_{\Phi_i^x}(P)\cap [1,n]$  for $n\in \mathfrak n_M$ : 
\hspace{-0,4cm}\begin{align}\label{ddf}\left(\int \sharp F(x)\, d\mathrm{Leb}_{D}^{A_n^M}(x)\right)\frac{H_{\nu_n^{M,N,P}}(Q^m)}{nm}&\geq -\frac{1}{n}\int \log \mathrm{Leb}_{D}^{A_n^M}\left(Q^{F(x)}(x)\right)\, d\mathrm{Leb}_{D}^{A_n^M}(x)\nonumber \\ &-\frac{H_{\mathrm{Leb}_{D}^{A_n^M}}(F)}{n}- \frac{3m\log \sharp P }{n}\int \sharp \partial F(x) \,d\mathrm{Leb}_{D}^{A_n^M}(x).
\end{align}

Observe that $ \int \sharp F(x)\, d\mathrm{Leb}_{D}^{A_n^M}(x)$ is just $n\mu_n^{M,N,P}(\mathbf M)$.  Moreover $F(x)\subset [1,n]$ and $\partial F(x)\subset 
\left([1,n]\cap \partial E^{\delta}_{\Phi_i^x}(P)\right)\cup \left([1,n]\cap \partial F^{\delta,M}_{-\Phi_{i+1}^x}(M)\right)\cup \left([1,n]\cap \partial G^{\delta,M}_{-\Phi_{i+1}^x}(N)\right)$, thus we have 
\begin{align*}
\sharp \partial F(x)&\leq  \lceil n/P\rceil+\lceil n/N \rceil + \lceil n/M \rceil= :a_n^{M,N,P}.
\end{align*}  
As $\sharp \partial F(x)$ completely determines $F(x)$, the number of possible values of $F$ is less than $ b_n^{M,N,P}:= \sum_{k=1}^{a_n^{M,N,P}}\binom{n}{k}$,  
 then $$H_{\mathrm{Leb}_{D}^{A_n^M}}(F)\leq \log  b_n^{M,N,P}.$$

We write $o_{M,N,P}(1)$ for any function $f$ of $M,N,P$ satisfying $$\limsup_P\limsup_N\limsup_{M}|f(M,N,P)|=0.$$  By a standard application of Stirling's formula,   we have  $\limsup_n\frac{1}{n}  \log b_n^{M,N,P}=o_{M,N,P}(1)$. 
Therefore  we obtain  by taking the limit when $n\in \mathfrak n_M$ goes to infinity in (\ref{ddf}): 
\begin{align*}\mu^{M,N,P}(\mathbf M)\frac{H_{\nu^{M,N,P}}(Q^m)}{m}& \geq & \\
& \hspace{-1,2cm} \liminf_{n\in \mathfrak n_M} -\frac{1}{n}\int \log \mathrm{Leb}_{D}^{A_n^M}\left(Q^{G_{-\Phi_{i+1}^x}^{\delta,M}\left(\left(N\right)\right)\cap E^{\delta}_{\Phi_i^x}(P)\cap [1,n]}(x)\right)\, \mathrm{Leb}_{D}^{A_n^M}(x)&\\
& \hspace{2cm}+o_{M,N,P}(1). &
\end{align*}
Let $E_{n}^{M,N,P}$ be the union of $]k,l]$ with $k,l\in E^\delta_{\Phi_i^x}$ and $]k,l]\subset G_{-\Phi_{i+1}^x}^{\delta,M}\left(\left(N\right)\right)\cap E^{\delta}_{\Phi_i^x}(P)\cap [1,n]$. Then 
one easily checks that  \begin{equation}\label{eq:almost}\overline{d}\left(\left( G_{-\Phi_{i+1}^x}^{\delta,M}\left(\left(N\right)\right)\cap E^{\delta}_{\Phi_i^x}(P)\cap [1,n]\right)\setminus E_{n}^{M,N,P}\right)\leq P/N +P/M.
\end{equation}

By Lemma \ref{dist} we get : 
\begin{align*} \liminf_{n\in \mathfrak n_M} -\frac{1}{n}\int \log \mathrm{Leb}_{D}^{A_n^M}\left(Q^{E_{n}^{M,N,P}}(x)\right) \,  d\mathrm{Leb}_{D}^{A_n^M}(x)\geq & & \\ \liminf_{n\in \mathfrak n_M} -\frac{1}{n}\int \log \mathrm{Leb}_{D}\left(Q^{E_{n}^{M,N,P}}(x)\right) \,  d\mathrm{Leb}_{D}^{A_n^M}(x)-\limsup_{n\in \mathfrak n_M}\frac{1}{n}\log \mathrm{Leb}_{D}(A_n^M)\geq & &\\
  \liminf_{n\in \mathfrak n_M}\int \int \psi_i \, d \mu_x^n[ E_{n}^{M,N,P} ]\, d \mathrm{Leb}_{D}^{A_n^M} +o_{M,N,P}(1)-\epsilon,& &\end{align*}
therefore 
$$\mu^{M,N,P}(\mathbf M)\frac{H_{\nu^{M,N,P}}(Q^m)}{m}\geq  \liminf_{n\in \mathfrak n_M}\int \int \psi_i \, d \mu_x^n[ E_{n}^{M,N,P} ]\, d \mathrm{Leb}_{D}^{A_n^M}+o_{M,N,P}(1)-\epsilon.$$
But it follows from (\ref{eq:almost}) that 
$$\liminf_{n\in \mathfrak n_M}\int \int \psi_i \, d \mu_x^n[ E_{n}^{M,N,P} ]\, d \mathrm{Leb}_{D}^{A_n^M}=\int \psi_i\, d\mu+o_{M,N,P}(1).$$

 Recall that the static entropy $\mathcal M(X)\ni \iota\mapsto H_{\iota}(R)$ is continuous at $\mu$ for any partition $R$ with boundary of   zero $\mu$-measure.  As the boundary of $Q$ has zero measure for  $\nu$, $\nu^{P},\nu^{ N,P}$ we get by taking the successive limits in $M$,  $N$ and $P$ : 
$$\mu(\mathbf M)\frac{H_{\nu}(Q^m)}{m}\geq \int \psi_i\,d\mu-\epsilon,$$
thus 
$$\frac{H_{\nu}(Q^m)}{m}\geq \int \psi_i\,d\nu-\epsilon/\lambda. $$
By letting $m$ go to infinity we obtain $h(\nu)\geq  h(\nu,Q)\geq \int \psi_i\,d\nu-\epsilon/\lambda$.  As it holds for any $\epsilon>0$, we have finally  $h(\nu)\geq   \int \psi_i\,d\nu$.  Since $\phi^i_{*}(x)>0>\phi_*^{i+1}(x)$ for $\nu$-a.e.  $x$,  the term $\int \psi_i\,d\nu$  is the integral of the sum of the positive exponents of $\nu$.  Together with Ruelle's inequality,  we conclude  that $\nu$ satisfies the  Pesin entropy formula,  thus  $\nu$ is a hyperbolic SRB measure. 

Then,   by standard arguments (see e.g.  \cite{bd}) it follows from the absolute continuity of Pesin stable lamination and  the dynamical density on hyperbolic times  with respect to $\mathcal D_i$ and $\Leb_D$   in $A_n$ that some point in $\mathcal D_i$ should lie in the basin of an ergodic component of $\nu$.  Therefore we get a contradiction : Lebesgue a.e.  $x\in \mathcal C_i\subset \bigcup_{\lambda, \delta\in \mathbb Q^+}\mathcal B_i(\lambda, \delta)$ lie in the basin of an ergodic hyperbolic SRB measure.   The proof of Proposition \ref{pro:mar} is complete.

\subsection{The non-hyperbolic case }
Let $\mathcal A'_i=\{ \alpha_i=1 >\alpha_{i+1} \text{ and }\beta_{i+1}=0\}$.   Then we have 
$\mathcal A_i\cup \mathcal A'_i=\mathcal L_i=\{\alpha_i=1 >\alpha_{i+1}\}$.     
\begin{prop}\label{secon}
 For  Lebesgue almost every $x\in \mathcal A'_i$ there is $\mu\in pw(x)$ with an SRB non-hyperbolic component. 
\end{prop}

This subsection is devoted to the proof of the above proposition. 
In the statement one may replace  $\mathcal A'_i$  by  a subset $\mathcal B'_i$ of $\mathcal A'_i$, such that $\alpha_i=1$  and   $\beta_{i+1}=0$  uniformly on $\mathcal C'_i$ and that  $\alpha_{i+1}(x)<1-\lambda $ for all $x\in \mathcal C'_i$ for some $\lambda>0$.  It is enough to show that for any subset $\mathcal D'_i$ of $\mathcal C'_i$ with positive Lebesgue measure  there is some $x\in \mathcal D_i$ and $\mu\in pw(x)$ with an SRB non-hyperbolic component.   Again,  we may choose  a smooth embedded  disc $D$ tangent to $\mathfrak C_i$ with $\Leb_{D}(\mathcal D'_i)>0$.  

\begin{lem}\label{lem:tech}For $\mathrm{Leb}_D$ a.e.  $x\in \mathcal D'_i$ and  for   any $\mathbb Q\ni\delta>0$,   there exist $\nu\in pw(x)$,  an infinite sequence $\mathfrak n$ and Borel subsets $(A_n)_{n\in \mathfrak n}$ of $\mathcal D'_i$ (depending on $x$ and $\delta$) with $\Leb_D(A_n)\geq \frac{1}{n^2}$ for all $n\in \mathfrak n$  such that we have :
\begin{itemize}
\item   $\sup_{y\in A_n}  \mathfrak d(\mu_n^y,\nu )\xrightarrow{\mathfrak n\ni n\to +\infty}0$, 
\item $\forall M \, \exists n_M \, \forall \mathfrak n\ni n >n_M \,  \forall  y\in A_n, \  d_n\left(\overline{E^\delta_{\Phi^{i+1}_y}(M)}\right)>\lambda.$
\end{itemize}
\end{lem}

We first recall a standard fact of measure theory. 
\begin{fact}Let $X$ and $Y$ be separable metric spaces. When  $\psi:X\rightarrow Y$ is a Borel map and $\mu$ is a Borel finite measure on $X$, there is a subset $X'$ of $X$ of  full $\mu$ measure such that 
for any $x\in X'$, for any $\delta>0$, the set $\psi^{-1}B(\psi(x), \delta)$ has positive $\mu$-measure.
\end{fact} 
The set $X'$ is  the preimage by $\psi$ of the essential range of $\psi$.  We refer to the appendix of  \cite{Bcmp} for a proof. 
\begin{proof}[Proof of Lemma \ref{lem:tech}]
 Fix $\delta>0$.  We apply the above fact to 
 \begin{itemize}
 \item $X=\mathbf M$, 
 \item $Y=\mathcal K\left(\mathcal{M}(\mathbf M,f)\times [0,1] \right)$  the set of compact subsets of $\mathcal{M}(\mathbf M,f)\times [0,1]$ endowed with the Hausdorff topology, 
 \item $\mu=\Leb_D(\mathcal D'_i\cap \cdot )$,
 \item   the map $\psi_M: \mathcal D'_i\rightarrow \mathcal K(\mathcal{M}(\mathbf M,f)\times [0,1] )$ which sends $x$ to  the set of accumulation points of the sequence $\left(\mu_x^n, d_n\left(\overline{E^\delta_{\Phi^{i+1}_x}(M)}\right)\right)_{n\in \mathbb N}$. 
 \end{itemize}
 Observe that for $\mu$-a.e. points $x$, there is $(\nu, \beta) \in \psi_M(x)\cap \left(pw(x)\times ]\lambda, +\infty[\right)$, because  for $x\in \mathcal D'_i$ we have
\begin{align*}
\lim_M^\searrow\overline{d}\left( \overline{E^\delta_{\Phi^{i+1}_x}(M)}\right)&\geq 1-\lim^\nearrow_M\underline{d}\left( E^\delta_{\Phi^{i+1}_x}(M)\right),\\
&\geq 
1-\alpha_{i+1}(x),\\
&>\lambda.
\end{align*}

For each $M$ we get therefore a subset  $\mathcal D''_{i,M}$ of $\mathcal D'_i$ with  $\Leb_D(\mathcal D'_i)=\Leb_D(\mathcal D''_{i,M})$ such for any $
x\in \mathcal D''_{i,M}$, there is $\nu_M \in pw(x)$ such that the set 
$$\left\{y\in \mathcal D'_i,\ \psi_M(y)\cap \left(B(\nu_M, 1/M)\times ]\lambda,+\infty[\right) \neq \emptyset  \right\}$$
has positive $\Leb_D$-measure.  By Borel-Cantelli lemma, for each $M$ there are infinitely many integers $n_M$ and Borel subsets $A_{n_M}$ of $D$ with  $\Leb_D(A_{n_M})\geq \frac{1}{n_M^2}$ such that  for all $y\in A_{n_M}$:
\begin{itemize}
\item   $\mathfrak d(\mu_{n_M}^y,\nu_M)\leq 1/M$, 
\item $ d_{n_M}\left(\overline{E^\delta_{\Phi^{i+1}_y}(M)}\right)>\lambda.$
\end{itemize}
  Let $x$ in the set $\bigcap_M \mathcal D''_{i,M}$ (which has full $\mathrm{Leb}_D$-measure in $\mathcal D'_i$) and let $\nu$ be a  limit of the above sequence $(\nu_M)_M$.  Finally we may choose $n_M$ as above such that $(n_M)_M$ is increasing.  This concludes the proof of the lemma with $\mathfrak n=(n_M)_{M}$. 
\end{proof}

From now we fix $x\in \mathcal D'_i$ satisfying the conclusions of Lemma \ref{lem:tech}.  Let   $\nu^\delta\in pw(x)$ and $(A_n^\delta)_{n\in \mathfrak n^\delta}$, $\delta\in \mathbb Q$,  be the  associated  measures and subsets given by this lemma.  We will show that any limit $\nu$ of $\nu^\delta$ when $\delta$ goes to zero has a non-hyperbolic  SRB  component $\hat \nu$. 
 We let for any  $n\in \mathfrak n^\delta$ and for any $M,P\in \mathbb N$
 $$\mu_n^{\delta,M,P}:=\int \mu_y^{n}\left[\overline{E^\delta_{\Phi^{i+1}_y}(M)}\cap E^\delta_{\Phi^{i}_y}(P) \right]\, d\mathrm{\Leb}_D^{A_n^\delta}(y),$$
 
 $$\zeta_n^{\delta,M}:=\int \mu_y^{n}\left[\overline{E^\delta_{\Phi^{i+1}_y}(M)}\right]\, d\mathrm{Leb}_D^{A_n^\delta}(y),$$
 
 $$\eta_n^\delta=\int \mu_y^{n}\, d\mathrm{Leb}_D^{A_n^\delta}(y).$$

By extracting subsequences,  we may assume $\mu_n^{\delta,M,P}$  (resp.  $\zeta_n^{\delta,M}$) is   converging to some $\mu^{\delta, M,P}$ (resp. $\zeta^{\delta,M}$)  when $n\in \mathfrak n^\delta$ goes to infinity.   Then  we let 

 $$\mu^{\delta,P}=\lim_M^{\searrow}\mu^{\delta,M,P},  \ \mu^{\delta}=\lim_P^\nearrow \mu^{\delta,P},  \ \hat\nu^\delta=\frac{\mu_\delta(\cdot)}{\mu_\delta(\mathbf M)},  \ \zeta^\delta=\lim_M^{\searrow}\zeta^{\delta,M}.$$  

These measures satisfy the following properties:
\begin{itemize}
\item $\zeta^{\delta}(\mathbf M)=\lim_M\lim_n\zeta_n^{\delta,M}(\mathbf M )=\lim_M \lim_n \int d_n\left(\overline{E^\delta_{\Phi^{i+1}_y}(M)}\right)\, d\mathrm{Leb}_{A_n^\delta}(y)\geq\lambda$,
\item  $\int \phi^{i+1}\, d\zeta^\delta\leq  \delta \zeta^{\delta}(\mathbf M)$ by Lemma \ref{one},
\item $\hat \nu^\delta$ is $f$-invariant and $h(\hat \nu^\delta)\geq \int \psi_i\, d\hat  \nu^\delta$ by arguing as in Subsection 3.1,
\item $\mathfrak d (\mu^\delta, \zeta^\delta)\xrightarrow{\delta}0$ because $\alpha_i=1$ uniformly on $\mathcal D'_i$.  Indeed we have by convexity of the distance $\mathfrak d$:
\begin{align*}
\mathfrak d (\mu^\delta, \zeta^\delta)&\leq \lim_P\lim_M\lim_n \mathfrak d (\mu_n^{\delta,M,P}, \zeta_n^{\delta,M}),\\
&\leq \limsup_P\limsup_n \int d_n\left(\overline{E^\delta_{\Phi^{i}_y}(P)}\right)\, d\mathrm{Leb}_D^{A_n^\delta}(y),\\
&\leq \limsup_P\limsup_n \sup_{x\in \mathcal D'_i}d_n\left(\overline{E^\delta_{\Phi^{i}_x}(P)}\right)\xrightarrow{\delta\rightarrow 0} 0.
\end{align*}
 
\item$\mathfrak d(\eta_n^\delta, \nu^\delta)\leq \int \mathfrak d ( \mu_y^{n},\nu^\delta) \, d\mathrm{Leb}_D^{A_n^\delta}(y)\xrightarrow{n}0$ according to the first item of Lemma \ref{lem:tech},
\item  $\mu^\delta\leq \zeta^\delta\leq \nu^\delta$.
\end{itemize}

\begin{lem}\label{lem:nonhyp}
 For any limit     $(\nu,\hat \nu)$ of $(\nu^\delta, \hat \nu^\delta)_\delta$  when $\delta$ goes to zero,  we have 

\begin{enumerate}
\item $\lambda \hat \nu\leq \nu$,
\item $\int \phi_{i+1}\, d\hat \nu\leq 0$,
\item $h(\hat \nu)\geq \int  \psi_i \, d\hat\nu$, 
\item $\hat \nu(\phi^{i+1}_*\geq 0)=1$.

\end{enumerate}
\end{lem}

\begin{proof}
Let $(\delta_k)_{k\in \mathbb N}$ be a sequence with $\lim_{k\to +\infty}\delta_k=0$ such that the measures  $\mu^{\delta_k}$, $
\nu^{\delta_k}$ and $
\hat\nu^{\delta_k}$ are converging respectively to $\mu$,  $\nu
$ and $\hat \nu
$  when $k$ goes to infinity. \\

\begin{enumerate}
\item As  $\mu^{\delta}$ is a component of $\nu^\delta$ and $\mathfrak d (\mu^\delta, \zeta^\delta)\xrightarrow{\delta\to 0}0$ we get at the limit 
\begin{align*}\lambda\hat \nu&\leq \lim_{k\to +\infty}\zeta^{\delta_k}(\mathbf M) \hat\nu^{\delta_k},\\
&\leq \lim_{k\to +\infty}\mu^{\delta_k}(\mathbf M) \hat \nu^{\delta_k},\\
&\leq \lim_{k\to +\infty}\mu^{\delta_k},\\
&\leq  \lim_{k\to +\infty}\nu^{\delta_k}=\nu.
\end{align*}\\

\item Observe that $\frac{\zeta_\delta(\cdot)}{\zeta_\delta(\mathbf M)} $ goes also to $\hat \nu$ with $\delta$,  since we have $\mathfrak d (\mu^\delta, \zeta^\delta)\xrightarrow{\delta\to 0}0$.  By taking the limit when $\delta$ goes to zero  in the inequality $\frac{1}{\zeta^{\delta}(\mathbf M)}\int \phi^{i+1}\, d\zeta^\delta\leq  \delta $, we get $\int \phi_{i+1}\, d\hat \nu\leq 0$.\\

\item The main result of \cite{fis} states that a partially hyperbolic system with a center bundle splitting in a dominated way into one dimensional subbundles is asymptotically $h$-expansive. In particular the measure theoretical entropy function is upper semicontinuous, therefore 
\begin{align*}
h(\hat \nu)&\geq \lim_{\delta\rightarrow 0} h(\hat \nu^\delta),\\
&\geq \lim_{\delta\rightarrow 0}\int \psi_i\, d\hat \nu^\delta=\int \psi_i\, d\hat \nu.
\end{align*}\\

\item  We may choose 
integers $n_k$ going to infinity with $k$ such that $\nu$ is the 
limit of $\left(\eta_{n_k}^{\delta_k}\right)_{k\in\mathbb N}$. As $\beta_{i+1}=0$ 
uniformly on $\mathcal D'_i$, we have $\nu(\phi^{i+1}_*\geq 0)=1$ by applying  Lemma \ref{two} with the sequence $(\xi_n)_n$ equal to $\left(\mathrm{Leb}_D^{A_{n_k}^{\delta_k}}\right)_{k\in \mathbb N}$. This concludes the proof of the last
item because we have shown $\lambda \hat \nu\leq \nu$.  \\
\end{enumerate}
\end{proof}

To conclude we only have to check $\hat \nu$ is a non-hyperbolic   SRB measure. From the two  items  (2) and (4) of Lemma \ref{lem:nonhyp} it follows that $\hat \nu\left(\phi_{i+1}^*=0\right)=1$. By the third item, the measure $\hat \nu$  satisfies Pesin   entropy formula. Therefore $\hat \nu$ is a non-hyperbolic   SRB component of $\nu\in pw(x)$. The proof of Proposition \ref{secon}, therefore of Theorem \ref{mmain}, is complete. \\

\begin{rem}If the center bundle is one dimensional,  then Lebesgue  a.e.  $x$ which does not lie in the basin of an ergodic hyperbolic SRB measure satisfies $\beta_1(x)=0$ by Proposition \ref{fond}.  Equivalently $\mu\left(\phi^1_*\geq 0\right)=1$  for any $\mu\in pw(x)$  by Lemma \ref{equiv}.  If $\mu\left(\phi^1_*> 0\right)=1$ for some $\mu\in pw(x)$,  then \begin{equation}\label{als}
0<\int \phi_1\, d\mu \leq \limsup_n\frac{1}{n}\sum_{l=0}^{n-1}\phi_1(f^lx).
\end{equation}
 But by  the main result of \cite{ADLP}  Lebesgue typical points satisfying (\ref{als}) lie in the basin of an ergodic hyperbolic SRB measure.  Therefore one recovers the main result of \cite{SDJ}, which states the limit $ \lim_n\frac{1}{n}\sum_{l=0}^{n-1}\phi_1(f^lx)$ defining the central exponent is well defined for Lebesgue a.e.  $x$ (equal to zero if and only if $x$ is not in the basin of an ergodic hyperbolic SRB measure).  
\end{rem}

If one assumes the partially hyperbolic to be only $\mathcal C^1$,  one gets the following version of Theorem \ref{mmain}:
\begin{theorem} Let $(f,\mathbf M)$ be a $\mathcal C^1$ diffeomorphism with a partially hyperbolic attractor admitting a center bundle splitting in a dominated way into one-dimensional subbundles, then for Lebesgue almost every $x$ in the topological basin of the attractor there is $\mu\in pw(x)$ with a  component satisfying the Pesin entropy formula.
\end{theorem}

\section{Proof of Proposition \ref{prop:second}}

We consider a sequence  $(\mu_q)_q$ of distinct ergodic SRB measures converging to some $\mu\in \mathcal M(\Lambda,f)$.  Note that $\mu$ is in general not ergodic.  We let  $i\in \{0,  \cdots,  k\}$ be the  (unique) integer with $\mu(\phi^i_*>0)=1$ and  $\mu(\phi^{i+1}_*>0)<1$. 

\subsection{The hyperbolic case}
We assume firstly $\mu(\phi^{i+1}_*< 0)>0$.   As $\mu_q$ are ergodic measures, the measurable function $\phi^i_*$ is constant $\mu_q$-a.e.  equal to $\int \phi^i\, d\mu_q$.  Fix $\delta_0\in ]0, \int \phi^i\,d\mu[$.  For $q$ large enough we have $\int \phi^i\, d\mu_q>\delta_0$.  Then we may choose $\delta_1>0$ so small that  $\lambda=\mu(\phi^{i+1}_*<-\delta_1)>0$.  Finally we let  $0<\delta<\min(\delta_0,\delta_1)$ such that $\mu(\phi^i_*<2\delta)<\lambda/4$.

For a subset $\mathrm E$ of $\mathbf M$ we let $\chi_{\mathrm E}$ be the indicator function of $\mathrm E$.  By Birkhoff ergodic Theorem for any $q$,  there is a set $\mathcal F_q$ of full $\mu_q$-measure such that for $x\in \mathcal F_q$ the empirical measures $\mu_{x}^n$,  $ \mu_{x}^n[F_{-\Phi_{i+1}^x}^{\delta,M}(M)]$, $ \mu_{x}^n[E^{\delta}_{\Phi_i^x}(P) ]$,  $\mu_{x_q}^n[F_{-\Phi_{i+1}^x}^{\delta,M}(M)\cap E^{\delta}_{\Phi_i^x}(P) ]$ are converging as follows.  Recall the sets $O_M$ and $H_\delta$ have been defined just after Lemma \ref{two}.  For $P\in \mathbb N^*$ we write $H_{\delta}^P(\phi^i)=\bigcup_{k=1,\cdots, P}T^{-k}H_\delta(\phi^i)$.  Then we have 
\begin{align*}\mu_{x}^n&\xrightarrow{n}\mu_q,\\
\mu_{x}^n[E^{\delta}_{\Phi_i^x}(P) ]&\xrightarrow{n}\chi_{H_{\delta}^P(\phi^i)}\mu_q=:\eta^{q,P},\\
 \mu_{x}^n[F_{-\Phi_{i+1}^x}^{\delta,M}(M)] &\xrightarrow{n} \chi_{O_M(-\phi^{i+1},\delta)}\mu_q=:\zeta^{q,M},\\
  \mu_{x_q}^n[F_{-\Phi_{i+1}^x}^{\delta,M}(M)\cap E^{\delta}_{\Phi_i^x}(P) ]&\xrightarrow{n}\chi_{O_M(-\phi^{i+1},\delta)\cap H_{\delta}^P(\phi^i)}\mu_q=:\mu^{q,M,P}
\end{align*}

By using (again) a Cantor diagonal argument we can assume the following successive limits exist :
\begin{align*}\mu^{M,P}&=\lim_q\mu^{q,M,P},  \ \mu_P=\lim_M \mu^{M,P}, \ \hat \mu=\lim_P\mu_P,\\
\eta^{P}&=\lim_q\eta^{q,P}, \ \eta=\lim_P\eta_P,\\
\zeta^{M}&=\lim_q\zeta^{q,M}, \ \zeta=\lim_M\zeta^{M}.
\end{align*} 
Observe that $\hat \mu$ is a component of $\eta$ and $\zeta$ which are both components of $\mu$. 
\begin{lem}
$$\hat \mu(\mathbf M)>\lambda/2.$$
\end{lem}

\begin{proof}Let $x\in \mathcal F_q$. The limit measure 
$\eta=\lim_P\lim_q \lim_n \mu_{x}^n[ E^{\delta}_{\Phi_i^x}(P) ]$, and the complement component $\overline{\eta}=\mu-\eta$ satisfy  $\int \phi^i_*\, d\overline{\eta}=\int \phi^i\,d\overline{\eta}=\delta\overline{\eta}(\mathbf M)$ and $\eta(\phi^i_*>\delta)=\eta(\mathbf M)$ (see Remark \ref{poli} and Remark \ref{rem:nouvpr}).  Since  $\phi^i_*(x)>0$  for $\mu$-a.e.   $x$,  therefore  for $\overline{\eta}$-a.e.  $x$, we have $\frac{1}{2\delta}\int \phi^i_*\,d\overline{\eta}\geq \overline{\eta}(\phi^i_*\geq 2\delta)$. 
Then we get \begin{align*}
\overline{\eta}(\phi^i_*<2\delta)&=\overline{\eta}(\mathbf M)-\overline{\eta}(\phi^i_*\geq 2\delta),\\
&\geq \overline{\eta}(\mathbf M)-\frac{1}{2\delta}\int \phi^i_*\,d\overline{\eta}, \\
&\geq \overline{\eta}(\mathbf M)/2.
\end{align*}
 But by assumption $\mu(\phi^i_*<2\delta)<\lambda/4$,  therefore $\overline{\eta}(\mathbf M)<\lambda/2$ and $\eta(\mathbf M)>1-\lambda/2$.

On the other hand the set $O_M=O_M(-\phi^{i+1},\delta)$ being open, the  limit $\zeta^M=\lim_q\chi_{O_M}\mu_q$ is larger than $
\chi_{O_M}\mu$.  But $O_M$ contains $\bigcup_{1\leq k\leq M}T^{-k}H_\delta(-\phi^{i+1})$ and $\mu\left(\{\phi^{i+1}_*<-\delta\}\setminus \left( \bigcup_{ k\in \mathbb N^*}T^{-k}H_\delta(-\phi^{i+1})\right)\right)=0$. Therefore any limit of $\chi_{O_M}\mu$ when $M$ goes to infinity is larger than $\chi_{\{\phi^{i+1}_*<-\delta\}}\mu$.  Consequently $\zeta(\mathbf M)$ is larger than $\mu(\phi^{i+1}_*<-\delta)>\lambda$.
 We conclude that 
\begin{align*}
\hat \mu(\mathbf M)&\geq \zeta(\mathbf M)+\eta(\mathbf M)-1,\\ 
&>\lambda/2.
\end{align*}

\end{proof}

\begin{rem}\label{rem:obs}
It follows from the first part of the above proof that the total mass of 
$\overline{\eta}=\overline{\eta}_\delta$ is less than $2\mu(\phi^i_*<2\delta)$, therefore goes to zero when $\delta$ goes to zero. 
\end{rem}

Recall the $i^{th}$ center exponent at $\mu$-typical points is positive.  Let $W^i(x)$ be the associated Pesin local manifold tangent to $G_i(x)$ at $\mu$-typical point $x$.  
For each $q$ we let $\nu_q$ be the conditional measure of $\mu_q$ on such  an unstable disc $D_q$.   We may assume $\Leb_{D_q}$ a.e. point $x$ lies in $\mathcal F_q$. Then we define 

$$\mu_n^{q,M,P}=\int \mu_x^n[F_{-\Phi_{i+1}^x}^{\delta,M}(M)\cap E^{\delta}_{\Phi_i^x}(P) ]\, d\mathrm{Leb}_{D_q}(y)$$
and
$$\nu_n^{q,M,P}=\frac{\mu_n^{q,M,P}}{\mu_n^{q,M,P}(\mathbf M)}. $$

Observe that $\mu_n^{q,M,P}$ (resp.  $\nu_n^{q,M,P}$) goes to
 $\mu^{q,M,P}$  (resp.  $\nu^{q,M,P}=\frac{\mu^{q,M,P}}{\mu^{q,M,P}(\mathbf M)}$) when $n$ goes to infinity.  Arguing as in the previous section,  for any $\epsilon>0$ we have for any partition $Q$ with diameter less than some $\alpha>0$ and for any $m\in \mathbb N^*$ :
$$\mu^{q,M,P}(\mathbf M)\frac{H_{\nu^{q,M,P}}(Q^m)}{m}\geq \int \psi_i\,  d\nu^{q,M,P}+o_{M,P}(1)-\epsilon, $$ 
where $o_{M,P}(1)$ denotes some function $f(M,P)$ satisfying  $\limsup_P\limsup_M |f(M,P)|=0$.   
With $\nu$ being  the limit measure $\nu = \lim_P\lim_M\lim_q\lim\nu^{q,M,P}=\frac{\hat \mu}{\hat\mu(\mathbf M)}$  we get in  the same way 
$$h(\nu)\geq \int \psi_i\, d\nu.$$
Moreover by applying   Lemma \ref{nouv} we have $\phi^{i+1}_*(x)\leq -\delta$ for $\zeta$ a.e.  $x$, therefore for $\nu$-a.e. $x$.  As $\hat \mu$ is a component of $\mu$ and $\mu(\phi^i_*>0)=1$, we have also $\phi^{i}_*(x)>0$  for $\nu$-a.e. $x$.   Thus  $\nu$ is a  hyperbolic SRB component of $\mu$ with unstable index $i$.  By using the absolute continuity of the stable foliation at $\hat \mu$ typical points one  concludes that a Lebesgue positive subset of $D_q$ is contained in the basin of an ergodic component $\xi$ of  $\nu$ for large $q$.  As  Lebesgue a.e.  point in $D_q$ is typical for $\mu_q$ we  conclude that $\xi=\mu_q=\nu=\mu$ for $q$ large enough.  

\begin{rem}
Contrarily to Proposition \ref{fond} we do not need here to work with midly hyperbolic times, because the measure $\mathrm{Leb}_{D_q}$ integrating the empirical measure  $ \mu_x^n[F_{-\Phi_{i+1}^x}^{\delta,M}(M)\cap E^{\delta}_{\Phi_i^x}(P) ]$ in the definition of $\mu_n^{q,M,P}$ does not depend on $M$.
\end{rem}

\subsection{The non-hyperbolic case}

We assume now $\mu(\phi_*^{i+1}\geq 0)=1$. We let again $D_q$ be a $\mu_q$-typical Pesin unstable disc of dimension $\mathrm{dim}(G_i)$ as above. 
We put
$$\mu_n^{q,\delta,M,P}=\int \mu_x^n[\overline{E_{\Phi_{i+1}^x}^{\delta}(M)}\cap E^{\delta}_{\Phi_i^x}(P) ]\, d\mathrm{Leb}_{D_q}(y),$$
$$\eta_n^{q,\delta,M}=\int \mu_x^n[\overline{E_{\Phi_{i+1}^x}^{\delta}(M)} ]\, d\mathrm{Leb}_{D_q}(y).$$
The measure $\eta_n^{q,\delta,M}$ is the same as $\eta_n^{q,M}$ in the above hyperbolic case.  However we write here  the dependence in $\delta$ as we will now make vary this parameter. 
Let $\mu^{q,\delta,M,P}$and $\eta^{q,\delta,M}$ be the limit of
$\mu_n^{q,\delta,M,P}$ and $\eta_n^{q,\delta,M}$ when $n$ goes to infinity.  The limits $\mu^\delta=\lim_P\lim_M\lim_q \mu^{q,\delta,M,P}$ and $\eta^\delta=\lim_M\lim_q\eta^{q,\delta,M}$, and the normalized probability $\nu^{\delta}=\frac{\mu^{\delta}}{\mu^{\delta}(\mathbf M)}$ are $f$-invariant and satisfy the following properties:

\begin{itemize}
\item $\mu^\delta$ is a component of $\eta^\delta$, which is itself a component of $\mu$,
\item  $\phi_*^{i+1}(x)\geq\delta$  for $\overline\eta^\delta$-a.e. $x$ with $\overline \eta_\delta=\mu-\eta^\delta$, therefore  $\eta^\delta(\mathbf M)\geq \mu(\phi_*^{i+1}=0 )>0$ and $\int \phi^{i+1} \, d\eta_\delta=\delta \eta_\delta(\mathbf M)$ (see Remark \ref{poli} and Remark \ref{rem:nouvpr} or  Proposition 6.2 of \cite{BCS}),
\item $h(\nu^\delta)\geq \int \psi_i \, d\nu^\delta$ by arguing as in Section 3.1,
\item $\mathfrak d(\mu^\delta, \eta^\delta)\xrightarrow{\delta\rightarrow 0}0$  by Remark \ref{rem:obs}.
\end{itemize}

\begin{lem}
 For any limit     $\hat \mu$ of $(\mu^\delta)_\delta$ when $\delta$ goes to zero, the associated probability $\hat \nu=\frac{\hat \mu}{\hat\mu(\mathbf M)}$  satisfies the following properties. 

\begin{enumerate}
\item $\int \phi^{i+1}\, d\hat \nu= 0$,
\item $\hat \nu\leq \mu$, in particular  $\hat \nu(\phi^{i+1}_*\geq 0)=1$,
\item $h(\hat \nu)\geq \int  \psi_i \, d\hat\nu$.
\end{enumerate}
\end{lem}

\begin{proof}
\begin{enumerate}
\item Observe firstly that $\frac{\eta_\delta}{\eta_\delta(\mathbf  M)}$ is converging to $\hat \nu$. Then by taking the limit when $\delta$ goes to zero in $\int \phi^{i+1} \, d\eta_\delta=\delta \eta_\delta(\mathbf M)$, we get the desired equality.
\item We have \begin{align*}
\mu(\phi_*^{i+1}=0 ) \hat \nu&\leq \lim_{\delta\to 0}\eta^\delta(\mathbf M)\nu^\delta,\\
&\leq \hat \mu,\\
&\leq \mu.
\end{align*} 
\item This follows from the aforementioned upper semicontinuity of the entropy function and the inequality $h(\nu^\delta)\geq \int \psi_i \, d\nu^\delta$ for any $\delta>0$. 
\end{enumerate}
\end{proof}

From the two first items we obtain $\phi^{i+1}_*(x)=0$ for  $\hat\nu$-a.e.  $x$.  
Then $\hat \nu$ satisfies the entropy formula by (3), therefore $\hat \nu$ is a non-hyperbolic SRB measure, which is a component of $\mu$ by (2).

\end{Large}

\end{document}